\documentclass[12pt]{article}
\usepackage{blindtext}
\usepackage[a4paper
, total={7in, 9in}
]{geometry}

\makeatletter

\@addtoreset{equation}{section}
\makeatother

\usepackage[cp1251]{inputenc}
\usepackage[OT1]{fontenc}
\usepackage[english]{babel}
\usepackage[tbtags]{amsmath}
\usepackage{amsfonts, amsmath, amsthm, amssymb, graphicx, comment}
\usepackage{cite}

\graphicspath{{figures/}}

\newcommand{\figout}[1]
{#1}

\newtheorem{lemma}{Lemma}
\newtheorem{theorem}{Theorem}

\newtheorem{corollary}{Corollary}

\newtheorem{remark}{Remark}

\theoremstyle{definition}

\def\a{\alpha}
\def\b{\beta}
\newcommand{\de}{\delta}
\newcommand{\D}{\Delta}
\def\lam{\lambda}
\def\g{\gamma}
\def\eps{\varepsilon}
\def\R{{\mathbb R}}
\def\then{\quad\Rightarrow\quad}

\def\vH{\vec H}
\def\vG{\vec G}

\def\L{\mathcal{L}}

\def\A{\mathcal{A}}
\def\tA{\widetilde{\A}}
\def\tlam{\widetilde{\lam}}
\def\td{\widetilde{d}}
\def\tq{\widetilde{q}}
\def\tit{\widetilde{t}}
\def\wht{\widehat{t}}
\def\hq{\widehat{q}}

\def\sL{sub-Lo\-rent\-zi\-an } 
\def\SL{Sub-Lo\-rent\-zi\-an } 

\newcommand{\Id}{\operatorname{Id}\nolimits}

\newcommand{\artanh}{\operatorname{artanh}\nolimits}

\newcommand{\spann}{\operatorname{span}\nolimits}
\newcommand{\Exp}{\operatorname{Exp}\nolimits}
\newcommand{\Lip}{\operatorname{Lip}\nolimits}
\newcommand{\sgn}{\operatorname{sgn}\nolimits}
\newcommand{\const}{\operatorname{const}\nolimits}
\newcommand{\VEC}{\operatorname{Vec}\nolimits}
\newcommand{\rank}{\operatorname{rank}\nolimits}
\newcommand{\intt}{\operatorname{int}\nolimits}
\newcommand{\cl}{\operatorname{cl}\nolimits}

\newcommand{\eq}[1]{$(\protect\ref{#1})$}
\newcommand{\be}[1]{\begin{equation}\label{#1}}
\newcommand{\ee}{\end{equation}}
\newcommand{\pder}[2]{\frac{\partial \, #1}{\partial \, #2} }
\newcommand{\map}[3]{#1 \, : \, #2 \to #3}
\newcommand{\restr}[2]{\left. #1 \right|_{#2}}

\newcommand{\onefiglabelsizen}[4]
{
\begin{figure}[htbp]
\begin{center}
\includegraphics[height=#4cm]{#1}
\\
\parbox[t]{0.7\textwidth}{\caption{#2}\label{#3}}
\end{center}
\end{figure}
}

\newcommand{\twofiglabelsizeh}[8]
{
\begin{figure}[htbp]
\includegraphics[height=#4cm]{#1}
\hfill
\includegraphics[height=#8cm]{#5}
\\
\parbox[t]{0.4\textwidth}{\caption{#2}\label{#3}}
\hfill
\parbox[t]{0.4\textwidth}{\caption{#6}\label{#7}}
\end{figure}
}

\author{Yu. L. Sachkov, E.F. Sachkova\\
Ailamazyan Program Systems Institute of RAS\\
Pereslavl-Zalessky,	 Russia\\
e-mail: yusachkov@gmail.com}

\title{Sub-Lorentzian distance and spheres \\on the Heisenberg group
\footnote{Sections 1, 2, 6--11 were written by Yu. Sachkov.
Sections 3--5 were written by E.~Sachkova. 
Work by Yu. Sachkov was supported by Russian Scientific Foundation, grant 22-11-00140, https://rscf.ru/project/22-11-00140/.
Work by E. Sachkova was supported by Russian Scientific Foundation, grant 22-21-00877, https://rscf.ru/project/22-21-00877/.
}
}

\begin{document}

\maketitle

\begin{abstract}
The left-invariant \sL problem on the Heisenberg group is considered. An optimal synthesis is constructed, the \sL distance and spheres are described.
\end{abstract}

%\newpage

\tableofcontents

%\newpage

\section{Introduction}
A sub-Riemannian structure on a smooth manifold $M$ is a vector distribution $\Delta \subset TM$ endowed with a Riemannian metric  $g$ (a positive definite  quadratic form). Sub-Riemannian geometry is a rich theory and an active domain of research during the last decades  
\cite{mont, jurd_book,  versh_gersh, notes, ABB, intro, UMN}. 

A \sL structure is a variation of a sub-Riemannian one for which the quadratic form $g$ in a distribution $\Delta$ is a Lorentzian metric (a nondegenerate quadratic form of index 1). \SL geometry tries to develop a theory similar to the sub-Riemannian geometry, and it is still in its childhood. For example, the left-invariant sub-Riemannian structure on the Heisenberg group is a classic subject covered in almost every textbook or survey on sub-Riemannian geometry. On the other hand, the left-invariant \sL structure on the Heisenberg group is not studied in detail. This paper aims to fill this gap.

The paper has the following structure.
In Sec. 2 we recall the basic notions of the \sL geometry. In Sec. 3 we state the left-invariant \sL structure on the Heisenberg group studied in this paper. Results obtained previously for this problem by M. Grochowski are recalled in Sec. 4. In Sec. 5 we apply the Pontryagin maximum principle and compute extremal trajectories; as a consequence, almost all extremal trajectories (timelike ones) are parametrized by the exponential mapping. In Sec. 6 we show that the exponential mapping is a diffeomorphism and find explicitly its inverse. On this basis in Sec. 7 we study optimality of extremal trajectories and construct an optimal synthesis.
In Sec. 8 we describe explicitly the \sL distance, in Sec. 9 we find its symmetries, and in Sec. 10 we study in detail the \sL spheres of positive and zero radii. Finally, in Sec. 11 we discuss the results obtained and pose some questions for further research.

\section{\SL geometry}\label{sec:SL}
A \sL structure on a smooth manifold $M$ is a pair $(\D, g)$ consisting of a vector distribution $\D \subset TM$ and a Lorentzian metric $g$ on $\D$, i.e., a nondegenerate quadratic form $g$ of index 1. \SL geometry attempts to transfer the rich theory of sub-Riemannian geometry (in which the quadratic form $g$ is positive definite) to the case of Lorentzian metric $g$. Research in \sL geometry was started by M. Grochowski \cite{groch2, groch3, groch4, groch6, groch9, groch11}, see also \cite{grong_vas, chang_mar_vas, kor_mar, groch_med_war}. 

Let us recall some basic definitions of \sL geometry.  
A vector $v \in T_qM$, $q \in M$, is called horizontal if $v \in \D_q$. A horizontal vector $v$ is called:
\begin{itemize}
\item
timelike if $g(v)<0$,
\item
spacelike if $g(v)>0$ or $v = 0$,
\item
lightlike if $g(v)=0$ and $v \neq 0$,
\item
nonspacelike if $g(v)\leq 0$.
\end{itemize}
A Lipschitzian curve in $M$ is called timelike if it has timelike velocity vector a.e.; spacelike, lightlike and nonspacelike curves are defined similarly.

A time orientation $X$ is an arbitrary timelike vector field in $M$. A nonspacelike vector $v \in \D_q$ is future directed if $g(v, X(q))<0$, and past directed if $g(v, X(q))>0$. 

A future directed timelike curve $q(t)$, $t \in [0, t_1]$, is called arclength paramet\-ri\-zed if $g(\dot q(t), \dot q(t)) \equiv - 1$. Any future directed timelike curve can be parametrized by arclength, similarly to the arclength parametrization of a horizontal curve in sub-Riemannian geometry.

The length of a nonspacelike curve $\g \in \Lip([0, t_1], M)$ is 
$$
l(\g) = \int_0^{t_1} |g(\dot \g, \dot \g)|^{1/2} dt.
$$

For points $q_1, q_2 \in M$ denote by $\Omega_{q_1q_2}$ the set of all future directed nonspacelike curves in $M$ that connect $q_1$ to $q_2$. In the case $\Omega_{q_1q_2} \neq \emptyset$ denote the \sL distance from the point $q_1$ to the point $q_2$ as
\be{d}
d(q_1, q_2) = \sup \{l(\g) \mid \g \in \Omega_{q_1q_2}\}.
\ee
Notice that in papers \cite{groch4, groch6} in the case $\Omega_{q_1q_2} = \emptyset$ it is set $d(q_1, q_2) = 0$. It seems to us more reasonable not to define $d(q_1, q_2)$ in this case.

A future directed nonspacelike curve $\g$ is called a \sL length maximizer if it realizes the supremum in \eq{d} between its endpoints $\g(0) = q_1$, $\g(t_1) = q_2$.

The causal future of a point $q_0 \in M$ is the set $J^+(q_0)$ of points $q_1 \in M$ for which there exists a future directed nonspacelike curve $\g$ that connects $q_0$ and $q_1$. The chronological future $I^+(q_0)$ of a point $q_0 \in M$ is defined similarly via future directed timelike curves $\g$. 

Let $q_0 \in M$, $q_1 \in J^+(q_0)$. The search for \sL length maximizers that connect $q_0$ with $q_1$ reduces to the search for future directed nonspacelike curves $\g$ that solve the problem
\be{lmax}
l(\g) \to \max, \qquad \g(0) = q_0, \quad \g(t_1) = q_1.
\ee

A set of vector fields $X_1, \dots, X_k \in \VEC(M)$ is an orthonormal frame for a \sL structure $(\D, g)$ if for all $q \in M$
\begin{align*}
&\D_q = \spann(X_1(q), \dots, X_k(q)),\\
&g_q(X_1, X_1) = -1, \qquad g_q(X_i, X_i) = 1, \quad i = 2, \dots, k, \\
&g_q(X_i, X_j) = 0, \quad i \neq j. 
\end{align*}
Assume that time orientation is defined by a timelike vector field $X \in \VEC(M)$ for which $g(X, X_1) < 0$ (e.g., $X = X_1$). Then the   \sL problem for the \sL structure with the orthonormal frame $X_1, \dots, X_k$ is stated as the following optimal control problem:
\begin{align*}
&\dot q = \sum_{i=1}^k u_i X_i(q), \qquad q \in M, \\
&u \in U = \left\{(u_1, \dots, u_k) \in \R^k \mid u_1 \geq \sqrt{ u_2^2 + \dots + u_k^2}\right\},\\
&q(0) = q_0, \qquad q(t_1) = q_1, \\
&l(q(\cdot)) = \int_0^{t_1} \sqrt{u_1^2 - u_2^2 -  \dots - u_k^2} \, dt  \to \max.
\end{align*}

\begin{remark}
The \sL length is preserved under monotone Lipschitzian time reparametrizations $t(s)$, $s \in [0, s_1]$. Thus if $q(t)$, $t \in [0, t_1]$, is a \sL length maximizer, then so is any its reparametrization $q(t(s))$, $s \in [0, s_1]$. 

In this paper we choose primarily the following parametrization of trajectories: the arclength parametrization ($u_1^2 - u_2^2 - \cdots - u_k^2 \equiv 1$) for timelike trajectories, and the parametrization with $u_1(t) \equiv 1$ for future directed lightlike trajectories. Another reasonable choice is to set $u_1(t) \equiv 1$ for all future directed nonspacelike trajectories.
\end{remark}

\section[Statement of the \sL problem on the Heisenberg group]{Statement of the \sL problem \\on the Heisenberg group}  
The Heisenberg group is the space 
  $M \simeq \R^3_{x,y,z}$  with the product rule
$$
(x_1, y_1, z_1) \cdot(x_2, y_2, z_2) = (x_1 + x_2, y_1 + y_2, z_1 + z_2 + (x_1y_2 - x_2 y_1)/2).
$$
It is a three-dimensional nilpotent Lie group with a left-invariant frame
\be{Xi}
X_1 = \pder{}{x} - \frac y2 \pder{}{z}, \qquad 
X_2 = \pder{}{y} + \frac x2 \pder{}{z}, \qquad
X_3 = \pder{}{z},
\ee
with the only nonzero Lie bracket $[X_1, X_2] = X_3$.

Consider the left-invariant \sL structure on the Heisenberg group $M$ defined by the orthonormal frame $(X_1, X_2)$, with the time orientation $X_1$. \SL length maximizers for this \sL structure are solutions to the optimal control problem
\begin{align}
&\dot q = u_1 X_1 + u_2 X_2, \qquad q \in M, \label{prf1} \\
&u \in U = \{(u_1, u_2) \in \R^2 \mid u_1 \geq |u_2|\}, \label{prf2} \\
&q(0) = q_0 = \Id = (0, 0, 0), \quad q(t_1) = q_1, \label{prf3}\\
&l(q(\cdot)) = \int_0^{t_1} \sqrt{u_1^2 - u_2^2} \, dt \to \max. \label{prf4} 
\end{align}

%\begin{comment}
Along with this (full) \sL problem, we will also consider a reduced \sL problem
\begin{align}
&\dot q = u_1 X_1 + u_2 X_2, \qquad q \in M, \label{pr21} \\
&u \in \intt U = \{(u_1, u_2) \in \R^2 \mid u_1 > |u_2|\}, \label{pr22} \\
&q(0) = q_0 = \Id = (0, 0, 0), \quad q(t_1) = q_1, \label{pr23}\\
&l(q(\cdot)) = \int_0^{t_1} \sqrt{u_1^2 - u_2^2} \, dt \to \max. \label{pr24} 
\end{align}
In the full problem \eq{prf1}--\eq{prf4} admissible trajectories $q(\cdot)$ are future directed nonspacelike ones, while in the reduced problem \eq{pr21}--\eq{pr24} admissible trajectories $q(\cdot)$ are only future directed timelike ones.
Passing to arclength-parametrized future directed timelike trajectories, we obtain a time-maximal problem equivalent to the reduced \sL problem \eq{pr21}--\eq{pr24}:
\begin{align}
&\dot q = u_1 X_1 + u_2 X_2, \qquad q \in M, \label{pr31} \\
&u_1^2 - u_2^2 = 1, \qquad u_1 > 0, \label{pr32} \\
&q(0) = q_0 = \Id = (0, 0, 0), \quad q(t_1) = q_1, \label{pr33}\\
&{t_1}  \to \max. \label{pr34} 
\end{align}

\section{Previously obtained results}\label{sec:groch}
The \sL problem on the Heisenberg group \eq{prf1}--\eq{prf4} was studied by M. Grochowski \cite{groch4, groch6}. In this section we present results of these works related to our results.

\begin{itemize}
\item[(1)]
\SL extremal trajectories were parametrized by hyperbolic and linear functions: were obtained formulas equivalent to our formulas \eq{qc=0}, \eq{qcn0}.
\item[(2)]
It was proved that there exists a domain in $M$ containing $q_0 = \Id$ in its boundary at which the \sL distance $d(q_0, q)$ is smooth.
\item[(3)]
The attainable sets of the \sL structure from the point $q_0 = \Id$ were computed: the chronological future of the point $q_0$
$$
I^+(q_0) = \{(x,y,z) \in M \mid -x^2 + y^2 + 4 |z|<0, \ x > 0\},
$$
and the causal future of the point $q_0$
\be{Jq0}
J^+(q_0) = \{(x,y,z) \in M \mid -x^2 + y^2 + 4 |z|\leq 0, \ x \geq  0\}.
\ee
In the standard language of control theory \cite{notes}, $I^+(q_0)$ is the attainable set of  the reduced system \eq{pr21}, \eq{pr22} from the point $q_0$ for arbitrary positive time. Thus the attainable set of the reduced system \eq{pr21}, \eq{pr22} from the point $q_0$ for arbitrary nonnegative time is 
$$
\A = I^+(q_0) \cup \{q_0\}.
$$
The attainable set of the full system \eq{prf1}, \eq{prf2} from the point $q_0$ for arbitrary nonnegative time is 
$$
\cl(\A) = J^+(q_0).
$$
The attainable set $\A$ was also computed in paper \cite{vinberg}, where its boundary  was called the Heisenberg beak. See the set $\partial \A$ in Figs. \ref{fig:beak}, \ref{fig:beak1}, and its views from the $y$- and $z$-axes in 
Figs.~\ref{fig:beaky} and \ref{fig:beakz} respectively.

\figout{
\onefiglabelsizen{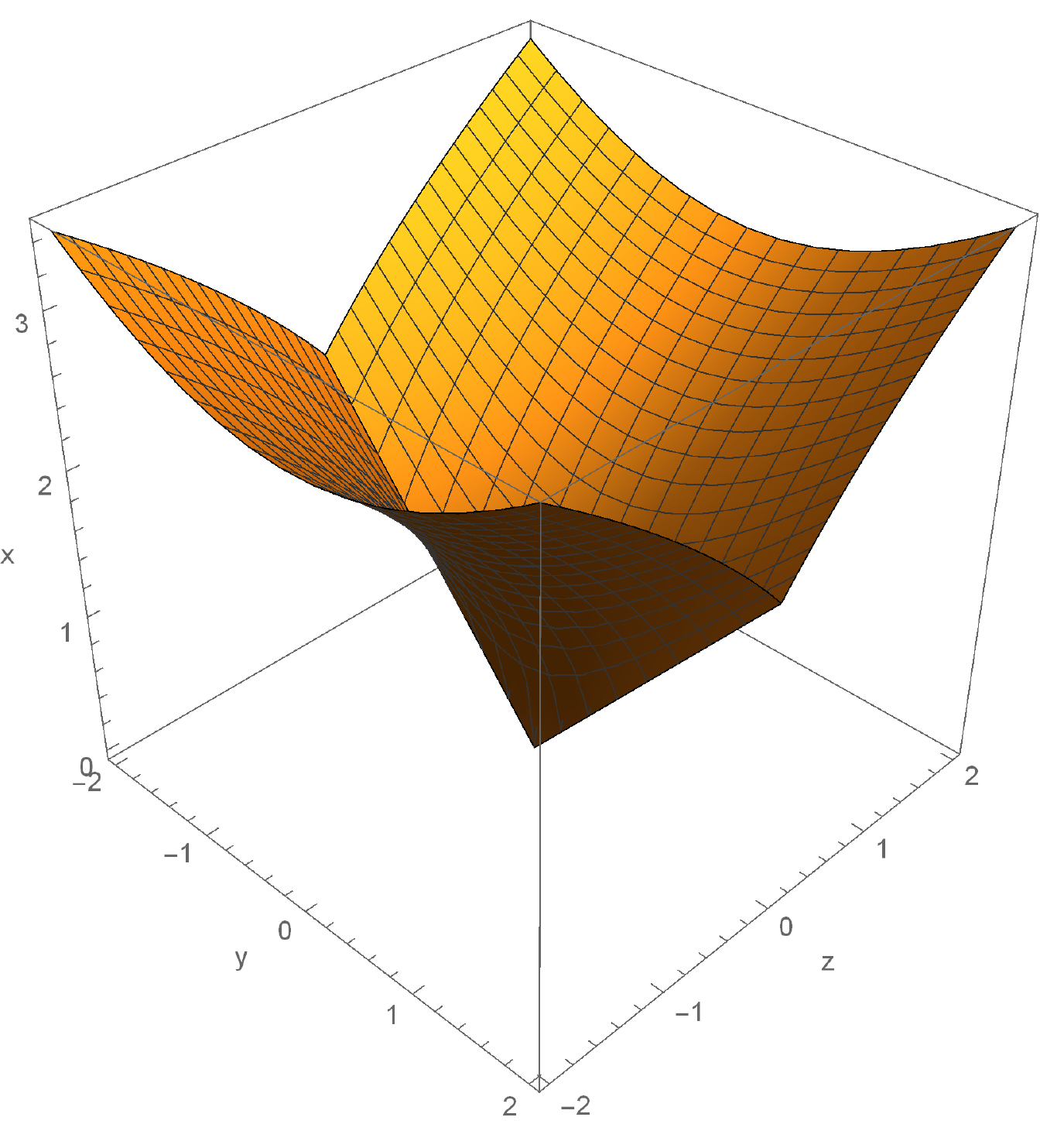}{The Heisenberg beak $\partial  \A$}{fig:beak}{8}

\twofiglabelsizeh
{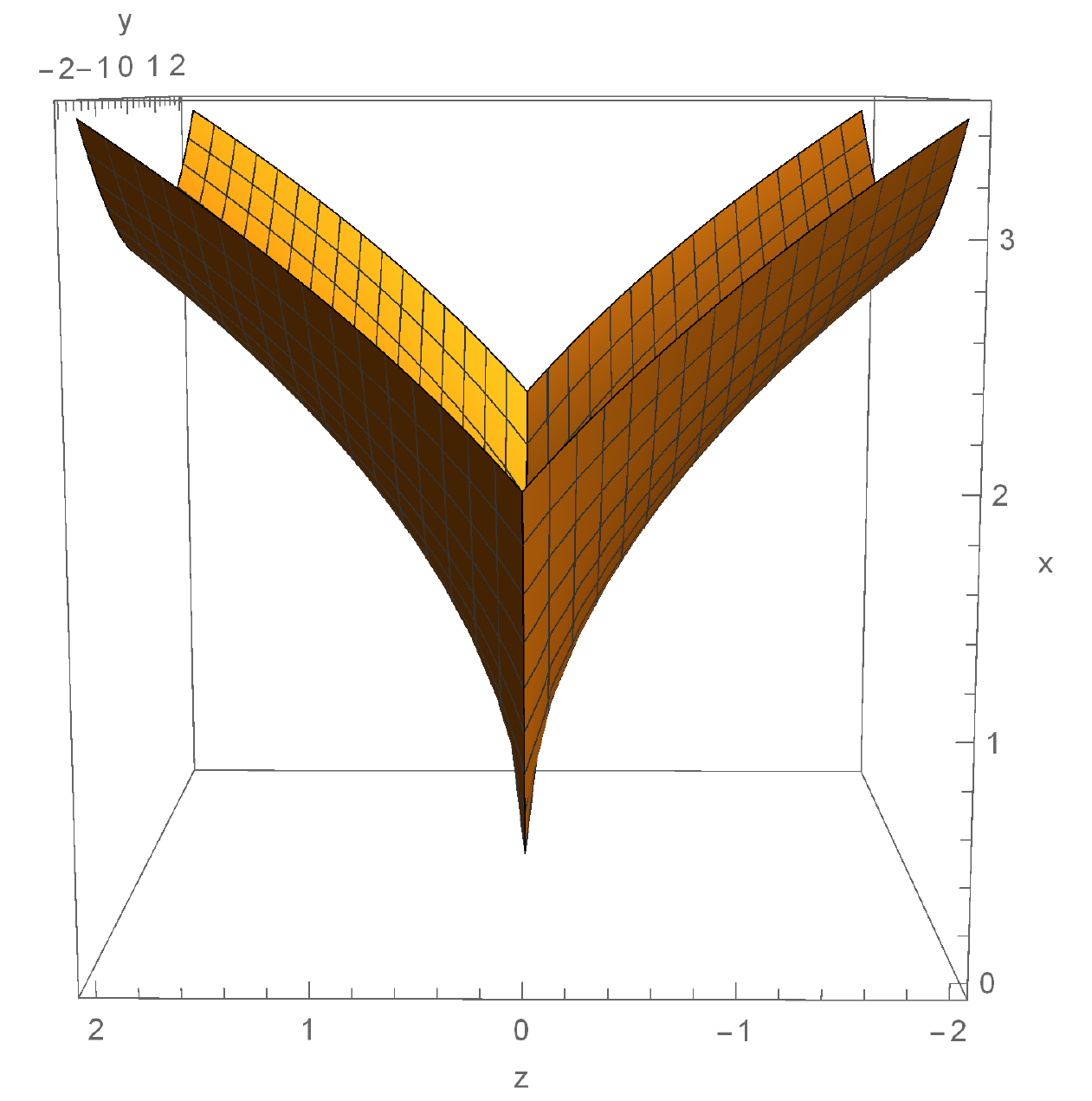}{View of $\partial  \A$ along $y$-axis}{fig:beaky}{6}
{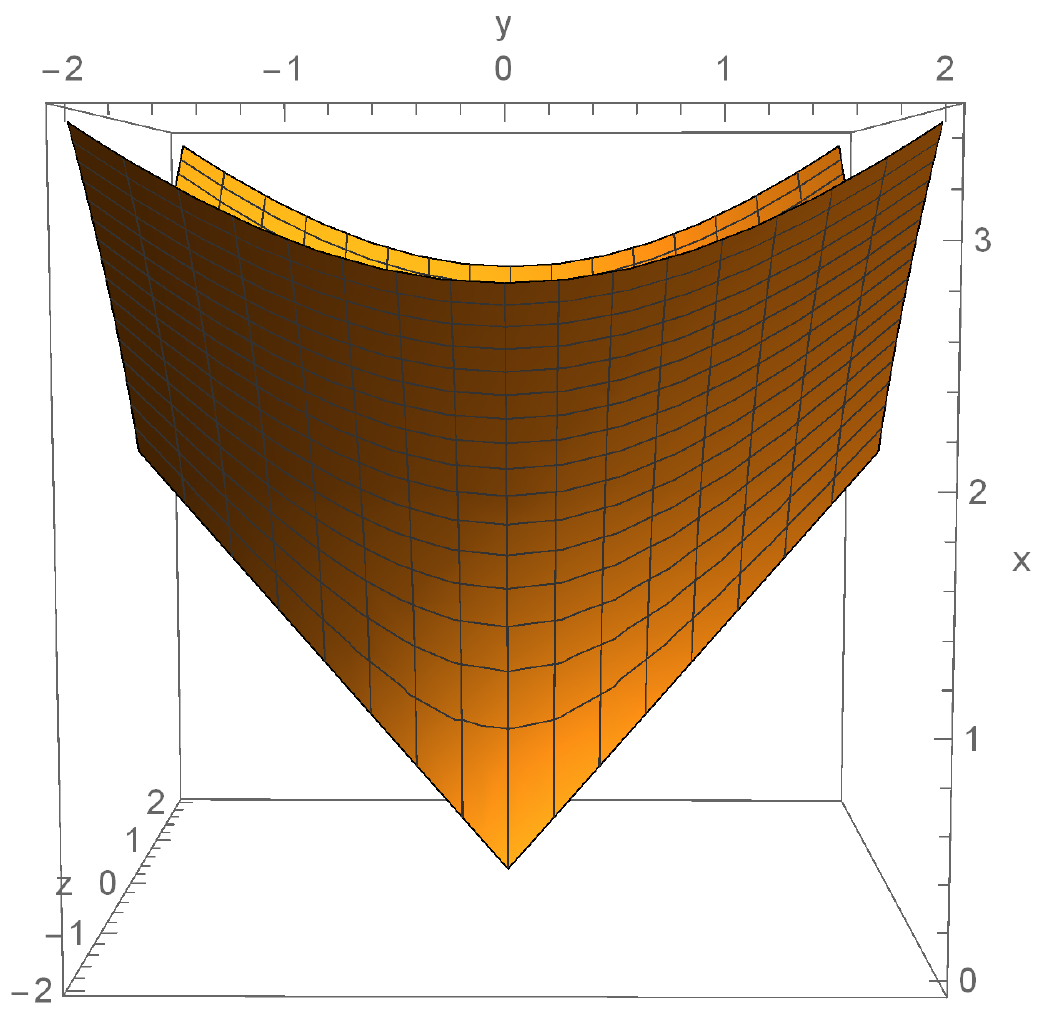}{View of $\partial  \A$ along $z$-axis}{fig:beakz}{6}
}
\item[(4)]
The lower bound of the \sL distance 
$$
\sqrt{x^2-y^2-4|z|} \leq d(q_0, q), \qquad q = (x, y, z) \in J^+(q_0),
$$
was proved. It was also noted that an upper bound $$d (q_0, q) \leq C \sqrt{x^2-y^2-4|z|} $$  does not hold for any constant $C \in \R$. 
\item[(5)]
It was proved that there exist non-Hamiltonian maximizers, i.e., maximizers that are not projections of the Hamiltonian vector field $\vH$, $H = \frac 12 (h_2^2 - h_1^2)$, related to the problem.
\end{itemize}

\section{Pontryagin maximum principle}\label{sec:PMP}
In this section we compute extremal trajectories of the \sL problem \eq{prf1}--\eq{prf4}.
The majority of results of this section were obtained by M. Grochowski \cite{groch4, groch6} in another notation, we present these results here for further reference.

Denote points of the cotangent bundle $T^*M$ as $\lam$. Introduce linear on fibers of $T^*M$ Hamiltonians $h_i(\lam) = \langle\lam, X_i\rangle$, $i = 1, 2, 3.$ Define the Hamiltonian of the Pontryagin maximum principle (PMP) for the \sL problem \eq{prf1}--\eq{prf4}
$$
h_u^{\nu}(\lam) = u_1 h_1(\lam) + u_2 h_2(\lam) - \nu \sqrt{u_1^2 - u_2^2}, \qquad \lam \in T^*M, \quad u \in U, \quad \nu \in \R.
$$
It follows from PMP \cite{PBGM, notes} that if $u(t)$, $t \in [0, t_1]$, is an optimal control in problem \eq{prf1}--\eq{prf4}, and $q(t)$, $t \in [0, t_1]$, is the corresponding optimal trajectory, then there exists a curve $\lam_{\cdot} \in \Lip([0, t_1], T^*M)$, $\pi(\lam_t) = q(t)$\footnote{where $\map{\pi}{T^*M}{M}$ is the canonical projection, $\pi(\lam) = q$, $\lam \in T^*_qM$}, and a number $\nu \in \{0, -1\}$ for which there hold the conditions for a.e. $t \in [0, t_1]$:
\begin{enumerate}
\item
the Hamiltonian system $\dot\lam_t = \vec{h}_{u(t)}^{\nu}(\lam_t)$\footnote{where $\vec{h}(\lam)$ is the Hamiltonian vector field on $T^*M$ with the Hamiltonian function $h(\lam)$},
\item
the maximality condition $h_{u(t)}^{\nu}(\lam_t) = \max_{v \in U} h_v^{\nu}(\lam_t) \equiv 0$,
\item
the nontriviality condition $(\nu, \lam_t) \neq (0, 0)$.
\end{enumerate}

A curve $\lam_{\cdot}$ that satisfies PMP is called an extremal, and the corresponding control $u(\cdot)$ and trajectory $q(\cdot)$ are called extremal control and trajectory.

\subsection{Abnormal case}

\begin{theorem}\label{th:abn}
In the abnormal case $\nu = 0$ extremals $\lam_t$ and controls $u(t)$ have the following form for some $\tau_1, \tau_2 \geq 0$:
\begin{itemize}
\item[$(1)$]
$h_3(\lam_t) \equiv \const > 0$:
\begin{align*}
t \in (0, \tau_1) &\then &&h_1(\lam_t) = h_2(\lam_t) < 0, \qquad &&u_1(t) = - u_2(t),\\
t \in (\tau_1, \tau_1+\tau_2) &\then &&h_1(\lam_t) = -h_2(\lam_t) < 0, \qquad &&u_1(t) =  u_2(t).
\end{align*}
\item[$(2)$]
$h_3(\lam_t) \equiv \const < 0$:
\begin{align*}
t \in (0, \tau_1) &\then &&h_1(\lam_t) = -h_2(\lam_t) < 0, \qquad &&u_1(t) =  u_2(t),\\
t \in (\tau_1, \tau_1+\tau_2) &\then &&h_1(\lam_t) = h_2(\lam_t) < 0, \qquad &&u_1(t) =  -u_2(t).
\end{align*}
\item[$(3)$]
$h_3(\lam_t) \equiv   0$:
\begin{align*}
&(h_1, h_2)(\lam_t) \equiv \const \neq (0, 0), \qquad h_1(\lam_t) \equiv -|h_2(\lam_t)|, \\
&u(t) \equiv \const, \qquad u_1(t) \equiv \pm u_2(t), \quad \pm = - \sgn(h_1h_2(\lam_t)).
\end{align*}
\end{itemize}
\end{theorem}
\begin{proof}
Apply the PMP for the case $\nu = 0$.
\end{proof}

\begin{corollary}
Along abnormal extremals $H(\lam_t) \equiv 0$, where $H = \frac 12(h_2^2 - h_1^2)$.
\end{corollary}

\subsection{Normal case}
In the normal case ($\nu = -1$)  extremals exist only for $h_1 \leq - |h_2|$.\footnote{The set $\{(h_1, h_2) \in (\R^2)^* \mid h_1 \leq - |h_2|\}$ is the polar set to $U$ in the sense of convex analysis.} In the case $h_1 = - |h_2|$ normal controls and extremal trajectories coincide with the abnormal ones.
And in the domain $\{ \lam \in T^*M \mid h_1 < - |h_2|\}$ extremals are reparametrizations of trajectories of the Hamiltonian vector field $\vH$ with the Hamiltonian $H = \frac 12(h_2^2 - h_1^2)$.
In the arclength parametrization, 
the  extremal controls are  
\be{u_norm}
(u_1, u_2)(t) = (-h_1(\lam_t), h_2(\lam_t)),
\ee 
and the 
  extremals satisfy the Hamiltonian ODE $\dot \lam = \vH(\lam)$ and belong to the level surface $\{H(\lam) = \frac 12\}$, in coordinates:
\begin{align*}
&\dot h_1 = - {h_2h_3}, \qquad \dot h_2 = -  {h_1h_3}, \qquad \dot h_3 = 0, \\
&\dot q = \cosh \psi \, X_1 + \sinh \psi \, X_2, \\
&h_1 = -  \cosh \psi, \qquad h_2 =   \sinh \psi,  \qquad \psi \in \R.
\end{align*}
We denote $c = h_3$ and obtain a parametrization of normal trajectories $q(t) = \pi \circ e^{t \vH}(\lam_0)$, $\lam_0 \in H^{-1}\left(\frac 12\right) \cap T^*_{q_0}M$.
If $c = 0$, then  
\be{qc=0}
x = t \cosh \psi, \quad y = t \sinh \psi, \quad z = 0.
\ee
If $c \neq 0 $, then  
\be{qcn0}
 x = \frac{\sinh(\psi + ct) - \sinh \psi}{c}, 
\quad
y = \frac{\cosh(\psi + ct) - \cosh \psi}{c}, 
\quad
z = \frac{\sinh(ct) - ct}{2c^2}.
\ee 

Summing up, we obtain the following characterization of normal trajectories in the   \sL problem \eq{prf1}--\eq{prf4}.

\begin{theorem}\label{th:norm}
Normal controls and trajectories either coincide with abnormal ones (in the case  $h_1(\lam_t) = - |h_2(\lam_t)|$, see Th. $\ref{th:abn}$), or can be arclength parametrized to get controls \eq{u_norm} and  future directed timelike trajectories \eq{qc=0} if $c = 0$, or \eq{qcn0} if $c \neq 0$.

In particular,
along each normal extremal $H(\lam_t) \equiv \const \in \left\{0, \frac 12 \right\}$.
\end{theorem}

Consequently, normal trajectories are either nonstrictly normal (i.e., simultaneously normal and abnormal) in the case $H = 0$, or strictly normal (i.e., normal but not abnormal) in the case $H = \frac 12$. Strictly normal arclength-parametrized trajectories are described by the exponential mapping
\begin{align}
&\map{\Exp}{N}{\tA}, \qquad (\lam, t) \mapsto q(t) = \pi \circ e^{t \vec{H}}(\lam), \label{Exp}\\
&N = C \times \R_+,  \qquad \R_+ = (0, + \infty), \qquad C =  T^*_{\Id} M \cap 
H^{-1}\left(\frac 12\right)
 \simeq \R^2_{\psi, c},  \nonumber\\
&\tA = \intt \A  = I^+(q_0) \nonumber
\end{align}
given explicitly by formulas \eq{qc=0}, \eq{qcn0}.

In papers \cite{groch4, groch6} were obtained formulas equivalent to \eq{qc=0}, \eq{qcn0}.

\begin{remark}
Projections of strictly normal (future directed timelike) trajectories to the plane $(x, y)$ are:
\begin{itemize}
\item
 either rays $y = k x$, $x \geq 0$, $k \in (-1, 1)$ (for $c=0$), see Fig. $\ref{fig:xyc0}$,
\item
 or arcs of hyperbolas  with asymptotes $x = \pm y > 0$ (for $c \neq 0$), see Fig. $\ref{fig:xycn0}$.
\end{itemize}

\figout{
\twofiglabelsizeh
{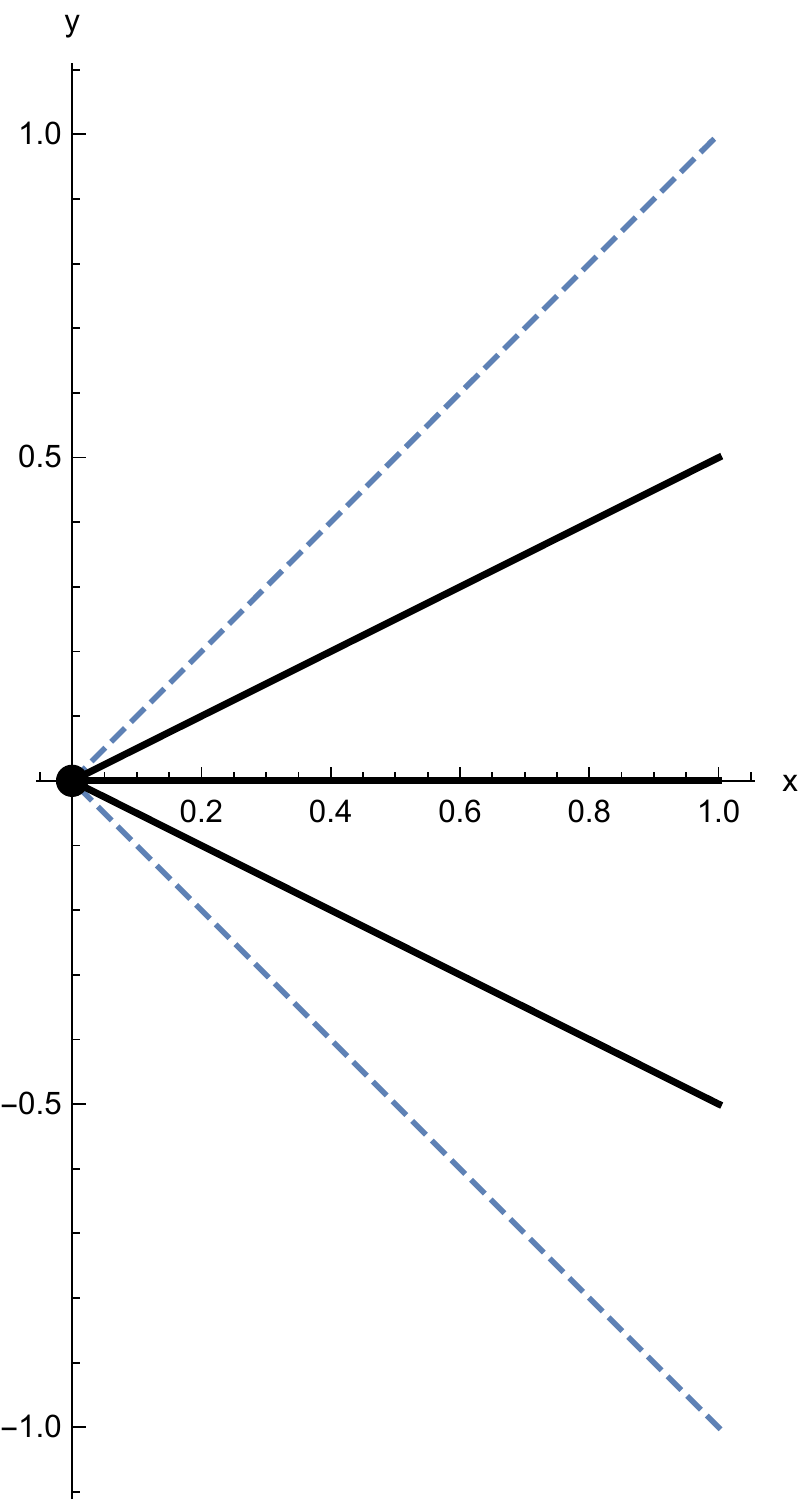}{Strictly normal $(x(t), y(t))$, $c = 0$}{fig:xyc0}{7}
{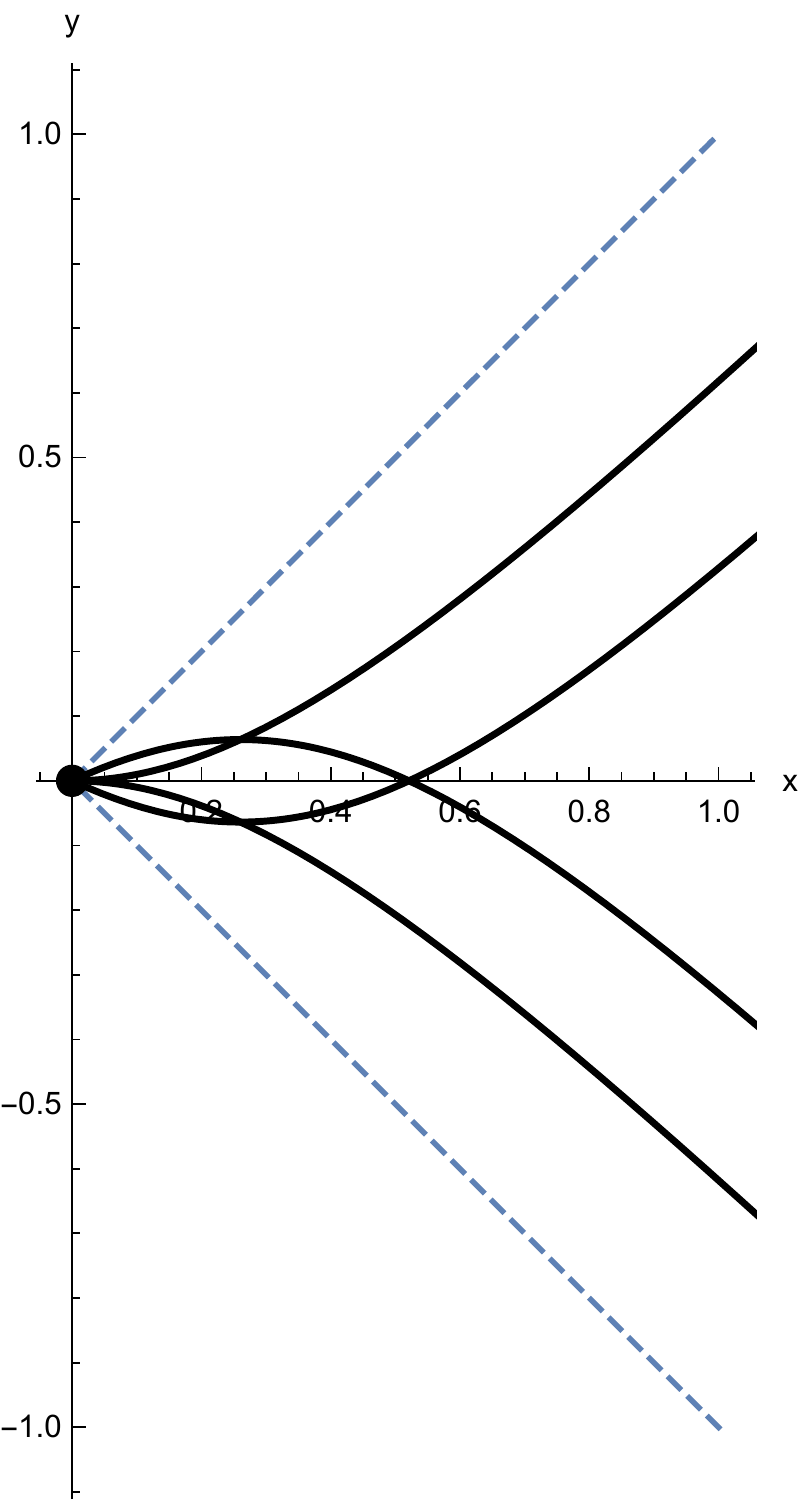}{Strictly normal $(x(t), y(t))$, $c \neq 0$}{fig:xycn0}{7}
}

Projections of nonstrictly normal (future directed lightlike) trajectories to the plane $(x, y)$ are broken lines with one or two edges parallel to the rays $x = \pm y > 0$, see Fig. $\ref{fig:xyabn}$.

\figout{
\onefiglabelsizen
{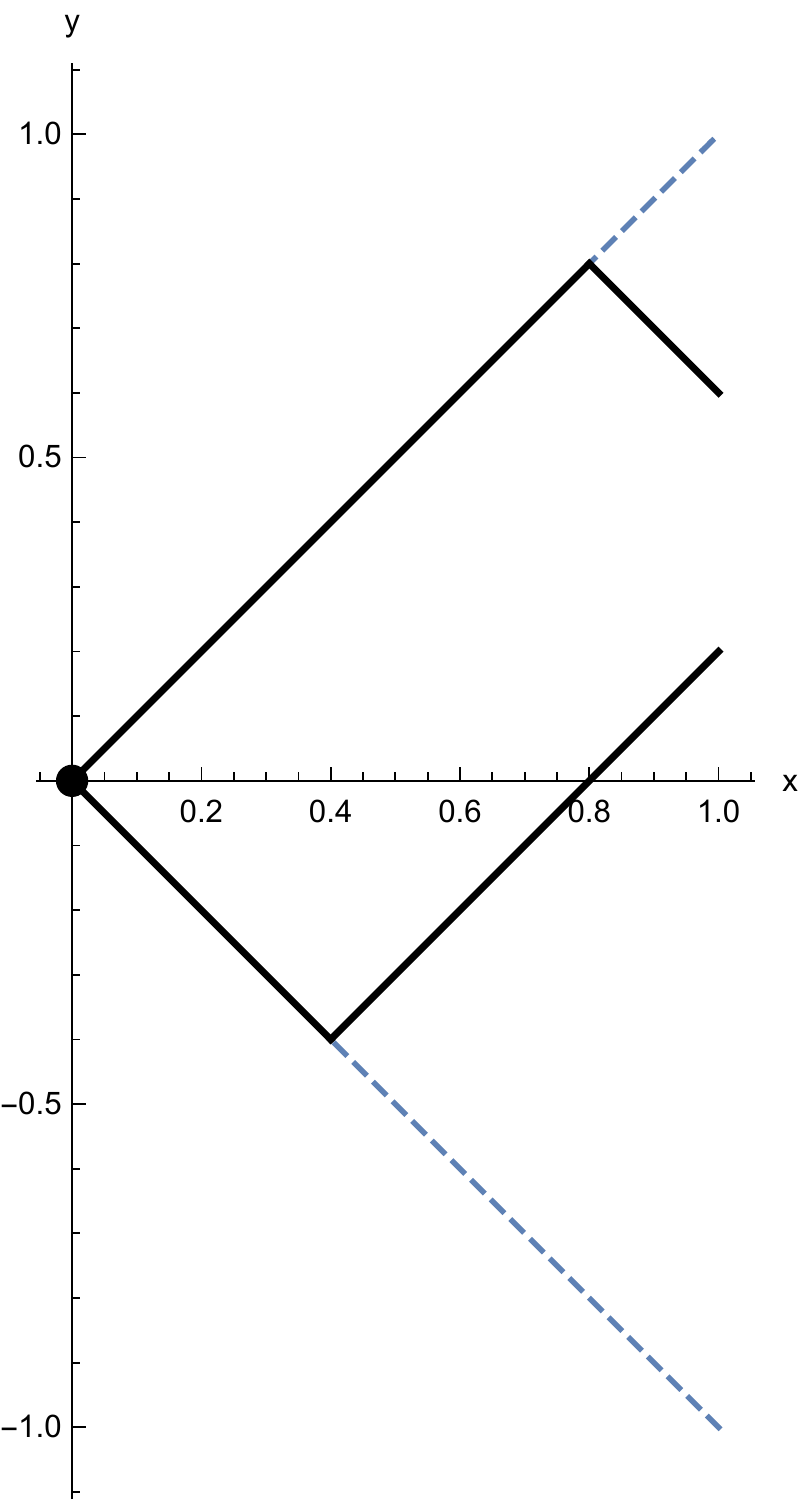}{Nonstrictly normal $(x(t), y(t))$}{fig:xyabn}{7}
}

Projections of all extremal trajectories (as well as of all admissible trajectories) to the plane $(x,y)$ are contained in the angle $\{(x, y) \in \R^2 \mid x \geq |y| \}$, which is the projection of the attainable set $J^+(q_0)$ to this plane. 
\end{remark}

\begin{remark}
The  Hamiltonian $H = \frac 12 (h_2^2-h_1^2)$ is preserved on each extremal. On the other hand, since the problem is left-invariant, the extremals respect the symplectic foliation on the dual of the Heisenberg Lie algebra $T^*_{\Id}M = \{(h_1, h_2, h_3)\}$ consisting of $2$-dimensional symplectic leaves $\{h_3 = \const \neq 0\}$ and $0$-dimensional leaves $\{ h_3 = 0, \ (h_1, h_2) = \const\}$. Thus projections of extremals to $T^*_{\Id}M = \{(h_1, h_2, h_3)\}$ belong to intersections of the level surfaces $\left\{H = \const \in \left\{0, \frac 12\right\}\right\}$ with the symplectic leaves:
\begin{itemize}
\item
branches of hyperbolas $h_1^2-h_2^2 = 1$, $h_1 < 0$, $h_3 \neq 0$,
\item
points $(h_1, h_2) = \const$,  $H \in \left\{0, \frac 12\right\}$, $h_1 \leq - |h_2|$, $h_3 = 0$,
\item
angles $h_1 = - |h_2|$,  $h_3 \neq 0$.
\end{itemize}
See Figs. $\ref{fig:h123norm}$, $\ref{fig:h123abnorm}$.
\end{remark}

\figout{
\twofiglabelsizeh
{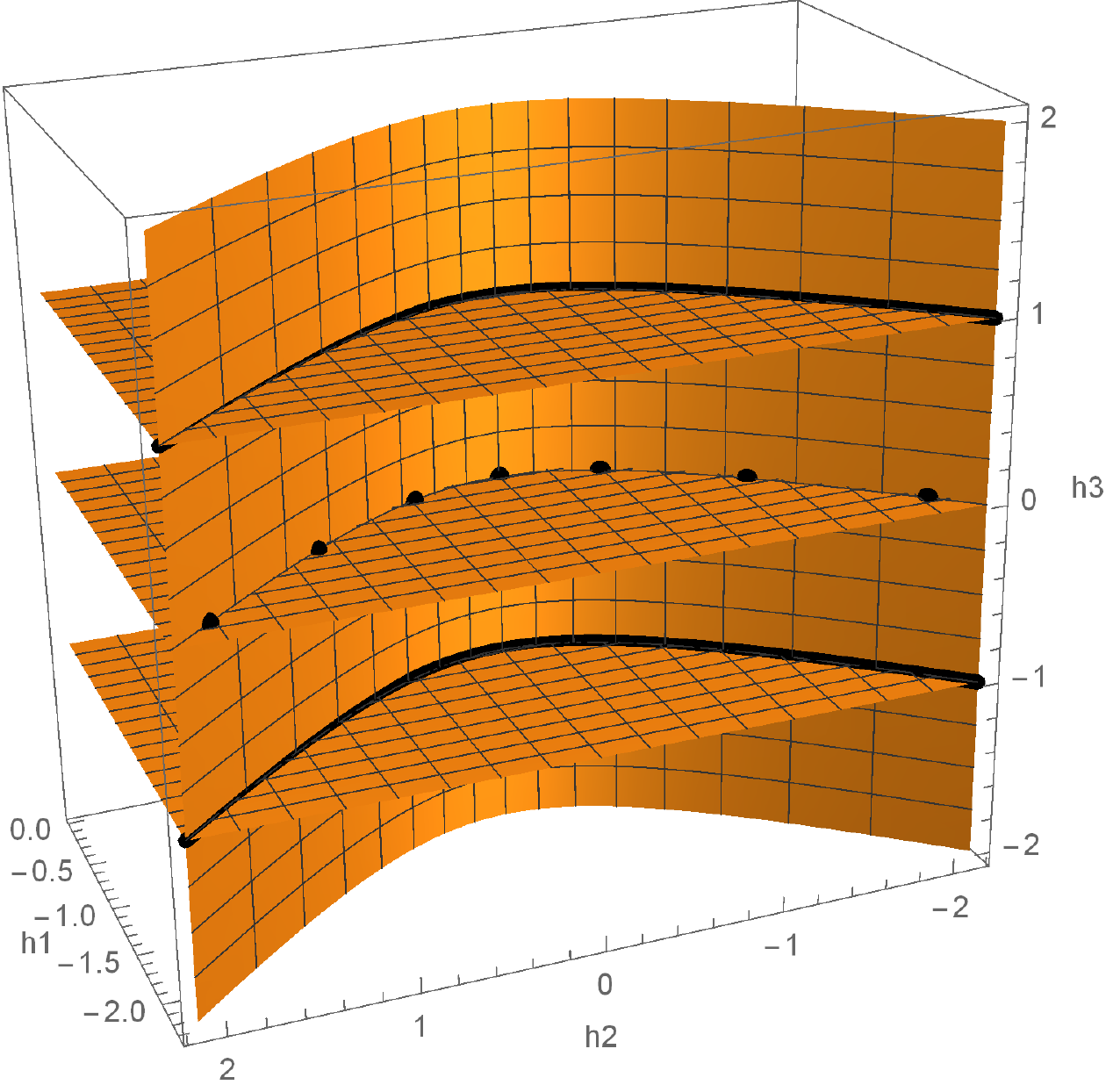}{Strictly normal $(h_1(t), h_2(t), h_3(t))$}{fig:h123norm}{6}
{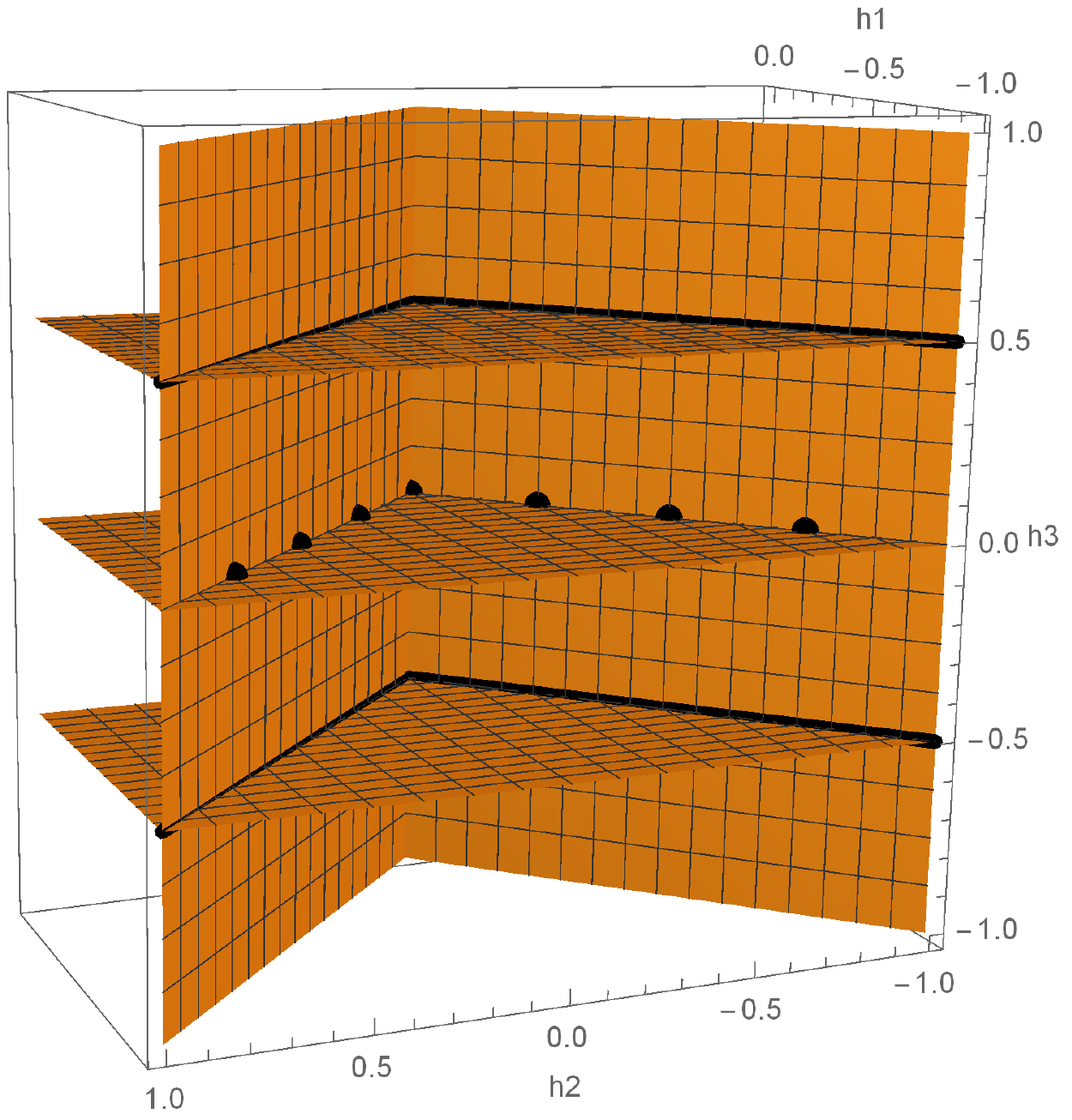}{Nonstrictly normal $(h_1(t), h_2(t), h_3(t))$}{fig:h123abnorm}{6}
}

\begin{remark}
In the sense of work {\em \cite{groch4}},
strictly normal extremal trajectories $q(t) = \pi\circ e^{t \vH}(\lam)$, $\lam \in C$, are Hamiltonian since they are projections of trajectories of the Hamiltonian vector field $\vH$. 

On the other hand, nonstrictly normal extremal trajectories given by items $(1)$, $(2)$ of Th. {\em\ref{th:abn}}  are non-Hamiltonian, e.g., the broken curves
\be{broken+-}
\begin{cases}
e^{t(X_1+X_2)}, & t \in [0, \tau_1],\\
e^{(t - \tau_1)(X_1-X_2)} \circ e^{\tau_1(X_1+X_2)}, & t \in [\tau_1, \tau_2],
\end{cases}
\ee 
and 
\be{broken-+}
\begin{cases}
e^{t(X_1-X_2)}, & t \in [0, \tau_1],\\
e^{(t - \tau_1)(X_1+X_2)} \circ e^{\tau_1(X_1-X_2)}, & t \in [\tau_1, \tau_2],
\end{cases}
\ee
for $0 < \tau_1 < \tau_2$. See item $(5)$ in Sec. {\em\ref{sec:groch}}. Although, each smooth arc of the broken trajectories \eq{broken+-}, \eq{broken-+} is a  reparametrization of projection of a trajectory of the Hamiltonian vector field $\vH$ contained in a face of the angle $\{(h_1, h_2, h_3) \in T_{\Id}^* M \mid h_1 = - |h_2|\}$, see Fig. $\ref{fig:h123abnorm}$.
\end{remark}

\section{Inversion of the exponential mapping}
 
\begin{theorem}\label{th:Exp-1}
The exponential mapping $\map{\Exp}{N}{\tA}$ is a real-analytic diffeomorphism. The inverse mapping   $\map{\Exp^{-1}}{\tA}{N}$, $(x, y, z) \mapsto (\psi, c, t)$, is given by the following formulas:
\begin{align}
&z = 0 \then \psi = \artanh \frac yx, \quad c = 0, \quad t = \sqrt{x^2-y^2}, \label{invz0}\\
&z \neq 0 \then \psi = \artanh \frac yx - p, \quad c = (\sgn z) \sqrt{\frac{\sinh 2p - 2p}{2z}},   \quad t = \frac{2p}{c},
\label{invzn0}
\end{align}
where $p = \b\left(\frac{z}{x^2-y^2}\right)$, and $\map{\b}{\left(-\frac 14, \frac 14\right)}{\R}$  is the inverse function to   the diffeomorphism
$$
\map{\a}{\R}{\left(-\frac 14, \frac 14\right)}, \qquad 
\a(p) = \frac{\sinh 2p-2p}{8 \sinh^2 p}.
$$
\end{theorem}

See plots of the functions $\a(p)$ and $\b(z)$ in Figs. \ref{fig:alpha} and \ref{fig:beta} respectively.
\figout{
\twofiglabelsizeh
{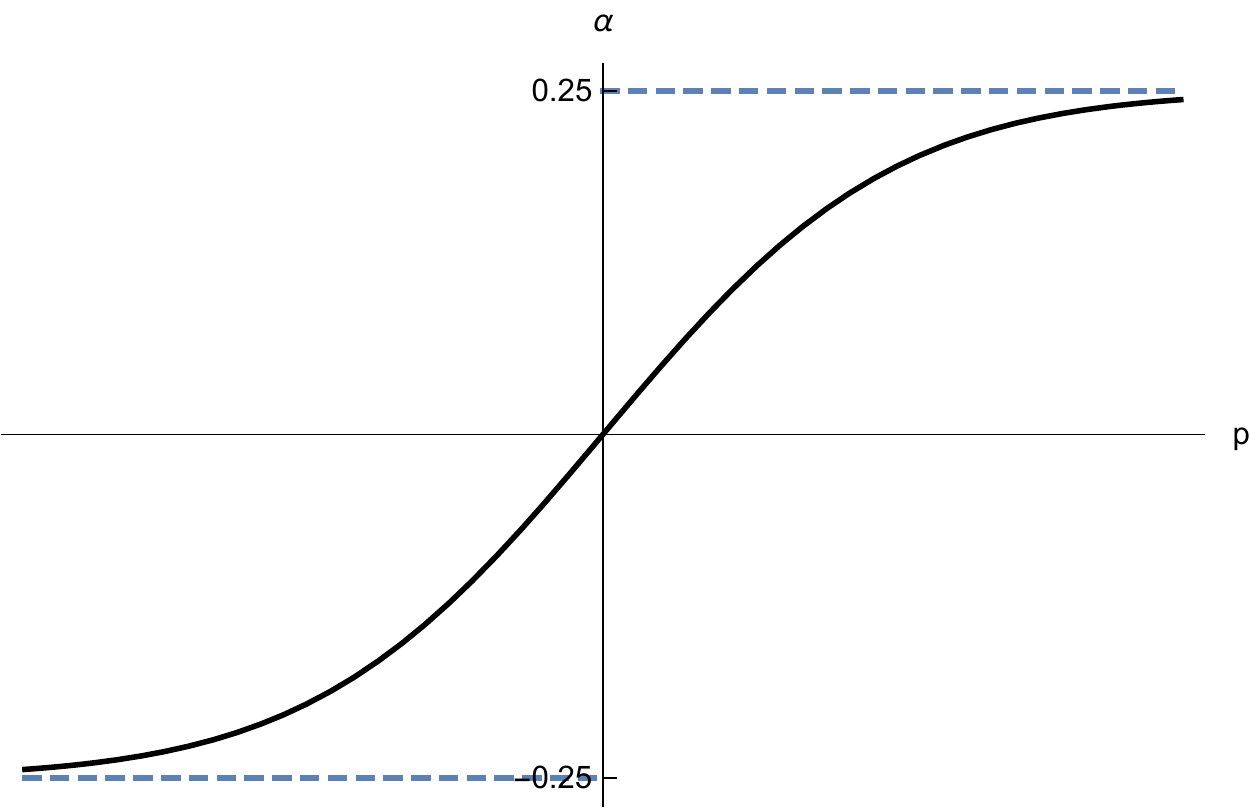}{Plot of $\a(p)$}{fig:alpha}{5}
{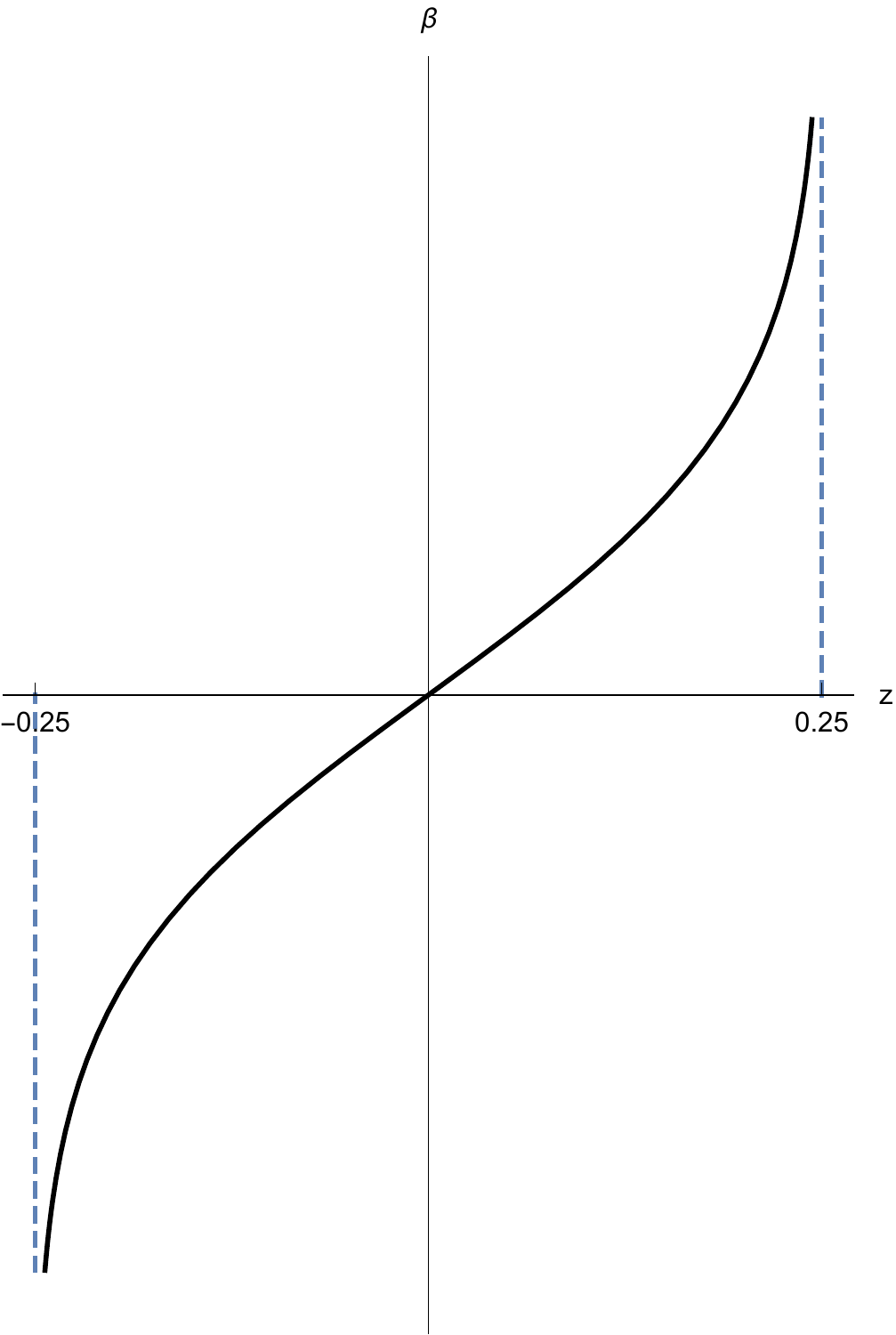}{Plot of $\b(z)$}{fig:beta}{8}
}

\begin{proof}
The exponential mapping is real-analytic since the strictly normal extremals are trajectories of the real-analytic Hamiltonian vector field $\vH$. We show that $\Exp$ is bijective.

Formulas \eq{invz0} follow immediately from \eq{qc=0}.

Let $c \neq 0$. Then formulas \eq{qcn0} yield
\begin{align}
&x = \frac 2c \sinh p \cosh \tau, \quad
y = \frac 2c \sinh p \sinh \tau, \quad
z = \frac{1}{2c^2}(\sinh 2p - 2 p), \label{xyznew}\\
&p = \frac{ct}{2}, \qquad \tau = \psi + \frac{ct}{2}. \label{ptau}
\end{align}
Thus
\begin{align}
&x^2 - y^2 = \frac{4}{c^2} \sinh^2 p, \label{x2-y2}\\
&\frac{z}{x^2-y^2} = \frac{\sinh 2p - 2p}{8 \sinh^2 p} = \a(p). \nonumber
\end{align}
The function $\a(p)$ is a diffeomorphism from $\R$ to $\left(-\frac 14, \frac 14\right)$, thus it has an inverse function, a diffeomorphism  $\map{\b}{\left(-\frac 14, \frac 14\right)}{\R}$. So $p = \b(\frac{z}{x^2-y^2})$. Now formulas \eq{invzn0} follow from \eq{xyznew}, \eq{ptau}.

So $\Exp$ is a smooth bijection with a smooth inverse, i.e., a diffeomorphism.
\end{proof}

 \section{Optimality of extremal trajectories}
We study optimality of extremal trajectories. The main tool is a 
sufficient optimality condition (Th. \ref{th:suf_opt}) based on a field of extremals (see \cite{notes}, Sec. 17.1).

We prove optimality of all extremal trajectories (Theorems \ref{th:distIf},  \ref{th:optJ}) without apriori theorem on existence of optimal trajectories. Such a theorem was recently proved \cite{lok_pod}, and it can shorten the proof of optimality in our work. 

\subsection{Sufficient optimality condition}
Let $M$ be a smooth manifold, then the cotangent bundle $T^*M$ bears the Liouville 1-form $s = p dq \in \Lambda^1(T^*M)$ and the symplectic 2-form $\sigma = ds = dp \wedge dq \in \Lambda^2(T^*M)$. 
A submanifold $\L \subset T^*M$ is called a Lagrangian manifold if $\dim \L = \dim M$ and   $\restr{\sigma}{\L} = 0$.

Consider an optimal control problem
\begin{align*}
&\dot q = f(q, u), \qquad q \in M, \quad u \in U,\\
&q(t_0) = q_0, \qquad q(t_1) = q_1, \\
&J[q(\cdot)] = \int_{t_0}^{t_1} \varphi(q, u) \, dt \to \min, \\
&t_0 \text{ is fixed}, \qquad t_1 \text{ is free}.
\end{align*}
Let $g_u(\lam) = \langle \lam, f(q, u)\rangle - \varphi(q, u)$, $\lam \in T^*M$, $q = \pi(\lam)$, $u \in U$, be the normal Hamiltonian of PMP. Suppose that the maximized normal  Hamiltonian $G(\lam) = \max_{u \in U} g_u(\lam)$ is smooth in an open domain $O \subset T^*M$, and let the Hamiltonian vector field $\vG \in \VEC(O)$ be complete. 

\begin{theorem}\label{th:suf_opt}
Let $\L \subset G^{-1}(0) \cap  O$ be a Lagrangian submanifold such that the form $\restr{s}{\L}$ is exact. Let the projection $\map{\pi}{\L}{\pi(\L)}$ be a diffeomorphism on a domain in $M$. Consider an extremal $\tlam_t = e^{t \vG}(\lam_0)$, $t \in [t_0, t_1]$, contained in $\L$, and the corresponding extremal trajectory $ \tq(t) = \pi(\tlam_t)$. Consider also any trajectory $q(t) \in \pi(\L)$, $t \in [t_0, \tau]$, such that $q(t_0) = \tq(t_0)$, $q(\tau) = \tq(t_1)$. Then $J[\tq(\cdot)] < J[q(\cdot)]$.  
\end{theorem}
\begin{proof}
Completely similarly to the proof of Th. 17.2 \cite{notes}.
\end{proof}

\subsection[Optimality in the reduced \sL problem]{Optimality in the reduced \sL problem \\on the Heisenberg group}
We apply Th. \ref{th:suf_opt} to the reduced \sL problem \eq{pr31}--\eq{pr34}.
For this problem
the maximized Hamiltonian $G = 1 - \sqrt{h_1^2 - h^2_2}$ is smooth on the domain $O = \{ \lam \in T^*M \mid h_1 < - |h_2| \}$, and the Hamiltonian vector field $\vG \in \VEC(O)$ is complete. In the domain $O$ the Hamiltonian vector fields $\vG$ and $\vH$ have the same trajectories up to a monotone time reparametrization; moreover, on the level surface $\left\{H = \frac 12\right\} = \{G = 0\}$ they just coincide between themselves. 

Define the set
\be{L}
\L = \left\{ e^{t \vG}(\lam_0) \mid \lam_0 \in C, \ t > 0\right\}.
\ee

\begin{lemma}
$\L \subset T^*M $ is a Lagrangian manifold such that $\restr{s}{\L}$ is exact.
\end{lemma}
\begin{proof}
Consider a smooth mapping
$$
\map{\Phi}{(T^*_{\Id}M \cap G^{-1}(0)) \times \R_+}{T^*M}, \qquad (\lam_0, t) \mapsto e^{t\vG}(\lam_0).
$$
Since
\begin{align*}
\rank \left( \pder{\Phi}{(t, \lam_0)}\right) 
&= \rank \left( \vG(\lam), e^{t \vG}_*\left(h_2 \pder{}{h_1} + h_1 \pder{}{h_2}\right),e^{t \vG}_*  \pder{}{h_3}\right) \\
&= \rank \left( \vG(\lam_0),  h_2 \pder{}{h_1} + h_1 \pder{}{h_2},  \pder{}{h_3}\right)\\
&= \rank \left( -h_1 X_1 + h_2 X_2,  h_2 \pder{}{h_1} + h_1 \pder{}{h_2},  \pder{}{h_3}\right)\\
&= 3,
\end{align*}
then $\L$ is a smooth 3-dimensional manifold.

Further,
$
\pi(\L) = \Exp(N) = \tA$ by Th. \ref{th:Exp-1}. Moreover, since $\Exp = \pi \circ \Phi$ and $\map{\Exp}{N}{\tA}$ is a diffeomorphism by Th. \ref{th:Exp-1}, then $\map{\pi}{\L}{\tA}$ is a diffeomorphism as well.

Let us show that $\restr{\sigma}{\L} = 0$. Take any $\lam = e^{t \vG}(\lam_0) \in \L$, $(\lam_0, t) \in N$, then $T_{\lam} \L = \R \vG(\lam) \oplus e^{t\vG}_* (T_{\lam_0}C)$. Take any two vectors $T_{\lam} \L \ni v_i = r_i \vG(\lam) + e^{t\vG}_* w_i$, $w_i \in T_{\lam_0}C$, $i = 1, 2$. Then
\begin{align*}
\sigma(v_1, v_2) = r_1 \sigma(\vG(\lam_0), w_2) + r_2 \sigma(w_1, \vG(\lam_0)) = 0
\end{align*}
since $\sigma(w_i, \vG(\lam_0)) = \langle dG, w_i \rangle = 0$ by virtue of $w_i \in T_{\lam_0} C = \{ dG = 0\}$.

So the 1-form $\restr{s}{\L}$ is closed. But $\tA$ is simply connected, thus $\L$ is simply connected as well. Consequently,  $\restr{s}{\L}$ is exact by the Poincar\'e lemma.
\end{proof}

\begin{theorem}\label{th:optI}
For any point $q_1 \in \intt \A = I^+(q_0)$  the strictly normal trajectory $q(t) = \Exp(\lam, t)$, $t \in [0, t_1]$, is the unique  optimal trajectory of the reduced \sL problem \eq{pr31}--\eq{pr34} connecting   $q_0$  with $q_1$, where $(\lam, t_1) = \Exp^{-1}(q_1) \in N$. 
\end{theorem}
\begin{proof} 
Take any $\lam_0 \in C$, $t_1 > t_0 > 0$. Then the Lagrangian manifold $\L$ \eq{L} and the  extremal $\tlam_t = e^{t\vG}(\lam_0)$, $t \in [t_0, t_1]$, satisfy hypotheses of 
  Th. \ref{th:suf_opt}. Thus the   trajectory $\tq(t) = \pi(\tlam_t)$, $t \in [t_0, t_1]$, is a strict maximizer for the reduced \sL problem \eq{pr31}--\eq{pr34}. 
	
	Take any $\lam_1 \in C$, $t_2 > 0$, and consider the extremal trajectory $\bar q(t) = \Exp(\lam_1, t)$, $t \in [0, t_2]$. Take any $\widehat q \in \tA$. The set $\A$ is an attainable set of a left-invariant control system on a Lie group, thus it is a semigroup. Consequently, $\widehat q \cdot \bar q(t)$ is an extremal trajectory contained in $\tA$. By the previous paragraph, this trajectory is a strict maximizer for the reduced \sL problem \eq{pr31}--\eq{pr34}. By left invariance of this problem, the same holds for the trajectory $\bar q(t) $, $t \in [0, t_2]$.
\end{proof}

Denote the cost function for the equivalent reduced \sL problems \eq{pr21}--\eq{pr24} and \eq{pr31}--\eq{pr34}:
\begin{align*}
\td(q_1) &= \sup \{ l(q(\cdot)) \mid \text{ traj. $q(\cdot)$ of \eq{pr21}--\eq{pr24}, $q(0) = q_0$, $q(t_1) = q_1$}\}\\
&=\sup \{ t_1 > 0 \mid \exists \text{ traj. $q(\cdot)$ of \eq{pr31}--\eq{pr34} s.t. $q(0) = q_0$, $q(t_1) = q_1$}\}, 
\end{align*}
where $q_1 \in \intt \A = I^+(q_0)$. This function has the following description and regularity property.

\begin{theorem}\label{th:distI}
Let  $q = (x, y, z) \in I^+(q_0)$. Then
\be{tdq}
\td(q) = \sqrt{x^2-y^2} \cdot \frac{p}{\sinh p}, \qquad p = \b\left(\frac{z}{x^2-y^2}\right). 
\ee

The function $\map{\td}{I^+(q_0)}{\R_+}$ is real-analytic.
\end{theorem}

\begin{proof}
Let $q \in I^+(q_0)$, then the \sL length maximizer from $q_0$ to $q$  for the reduced \sL problem \eq{pr31}--\eq{pr34} is described in Th. \ref{th:optI}, and the expression for $\td(q)$ in \eq{tdq} follows from the expression for $t$ in \eq{invzn0}. 

The both functions $\sqrt{x^2-y^2}$ and $\frac{p}{\sinh p}$ are real-analytic on $ I^+(q_0)$, thus $\td$ is real-analytic as well.
\end{proof}

\subsection[Optimality in the full \sL problem]{Optimality in the full \sL problem \\on the Heisenberg group}
In this subsection we consider the full \sL problem \eq{prf1}--\eq{prf4}.

\begin{theorem}\label{th:distIf}
Let  $q_1  \in I^+(q_0)$. Then
the \sL length maximizers for the full  problem \eq{prf1}--\eq{prf4}   are reparametrizations of the corresponding \sL length maximizer for the reduced   problem \eq{pr31}--\eq{pr34} described in Th. {\em \ref{th:optI}}. 

In particular, $\restr{d}{I^+(q_0)} = \td$.
\end{theorem}
\begin{proof}
Let $q(t)$, $t \in [0, t_1]$, be a trajectory of the full problem \eq{prf1}--\eq{prf4} such that $q(0) = q_0$, $q(t_1) = q_1$, and let $q(\cdot)$ be not a trajectory of the reduced   problem \eq{pr21}--\eq{pr24} (that is, there exist $0 \leq \tau_1 < \tau_2 \leq t_1$ such that $\restr{\left(u_1 - |u_2|\right)}{[\tau_1, \tau_2]} \equiv 0$).
Let $\tq(t)$, $t \in [0, \tit_1]$, be the optimal trajectory in the reduced   problem \eq{pr31}--\eq{pr34} connecting $q_0$ with $q_1$. We show that $l(q(\cdot)) < l(\tq(\cdot))$. By contradiction, suppose that $l(q(\cdot)) \geq l(\tq(\cdot))$.

Let $l(q(\cdot)) = l(\tq(\cdot))$. The trajectory $q(\cdot)$ does not satisfy the PMP for the full problem \eq{prf1}--\eq{prf4} (see Sec. \ref{sec:PMP}), thus it is not optimal in this problem. Thus there exists a trajectory $\bar q(\cdot)$ of this problem with the same endpoints and  $l(\bar q(\cdot)) > l(\tq(\cdot))$. The curve $\bar q(\cdot)$ cannot be a trajectory of the reduced system since its length is greater than the maximum $ l(\tq(\cdot))$ in this problem. So we can denote $\bar q(\cdot))$ as $q(\cdot)$ and assume that $l(q(\cdot)) > l(\tq(\cdot))$.

After time reparametrization we obtain that the control $u(t) = (u_1(t), u_2(t))$ corresponding to the trajectory $q(t)$, $t \in [0, t_1]$, satisfies $u_1(t) \equiv 1$, thus $|u_2(t)|\leq 1$. 

For any $\de \in (0, 1)$ define a function
$$
u_2^{\de}(t) = 
\begin{cases}
u_2(t) &\text{for } |u_2(t)| \leq 1 - \de, \\  
1-\de &\text{for } u_2(t) > 1 - \de, \\
\de-1 &\text{for } u_2(t) < \de-1,
\end{cases}
$$
so that 
\be{u2de}
|u_2^{\de}(t)| \leq 1 - \de, \quad |u_2^{\de}(t) - u_2(t)| \leq \de, \qquad t \in [0, t_1].
\ee
Define an admissible control $u^{\de}(t) = (1, u_2^{\de}(t))$, $t \in [0, t_1]$, and consider the corresponding trajectory $q^{\de}(t)$, $t \in [0, t_1]$, of the reduced problem \eq{pr21}--\eq{pr24}  with $q^{\de}(0) = q_0$. Denote its endpoint $q^{\de}(t_1) = q_1^{\de}$. By virtue of the second inequality in \eq{u2de},
\begin{align*}
&l(q^{\de}(\cdot)) = \int_0^{t_1} \sqrt{1 - \left(u_2^{\de}(t)\right)^2} dt \to \int_0^{t_1} \sqrt{1 - u_2^2(t)} dt = 
l(q(\cdot)), \\
&\max_{t \in [0, t_1]} \|q^{\de}(t) - q(t)\| \to 0
\end{align*}
as $\de \to + 0$.
So for sufficiently small $\de > 0$ we have
$$
l(q^{\de}(\cdot)) > l(\tq(\cdot)) \qquad \text{and} \qquad \|q_1^{\de} - q_1\| \text{ is small},
$$
where $\|\cdot\|$ is any norm in $M \cong \R^3$. In particular, $q_1^{\de} \in I^+(q_0)$ for small $\de>0$. 

Now let $\hq^{\de}(t)$, $t \in \left[0, \wht_1^{\de}\right]$, be the optimal trajectory  in the reduced problem \eq{pr31}--\eq{pr34} with the boundary conditions $\hq^{\de}(0) = q_0$, $\hq^{\de}\left(\wht_1^{\de}\right) = q_1^{\de}$. Then for   small $\de > 0$
\begin{align*}
&l\left(\hq^{\de}(\cdot)\right) \geq l(q^{\de}(\cdot)) > l(\tq(\cdot)), \\
&\left\|q_1^{\de} - q_1\right\| = \left\|\hq^{\de}\left(\wht_1^{\de}\right) - \tq(t_1)\right\| \text{ is small}. 
\end{align*}
By virtue of Th. \ref{th:distI}, the \sL distance $\map{\td}{I^+(q_0)}{\R_+}$ in the reduced problem \eq{pr31}--\eq{pr34}  is continuous, thus for   small $\de > 0$
$$
|l\left(\hq^{\de}(\cdot)\right) - l(\tq(\cdot))| = |\td(q_1^{\de}) - \td(q_1)|  \text{ is small}.
$$
Summing up, for   small $\de > 0$ the difference
$$
l(q(\cdot)) - l(\tq(\cdot)) < 
\left( l(q(\cdot)) - l\left(q^{\de}(\cdot)\right)\right) + 
\left( l\left(\hq^{\de}(\cdot)\right) - l\left(\tq(\cdot)\right)\right)
$$
becomes arbitrarily small, a contradiction.
Thus $\tq(\cdot)$ is optimal and $q(\cdot)$ is not optimal in the full \sL problem \eq{prf1}--\eq{prf4}.
\end{proof}

 \begin{theorem}\label{th:optJ}
Let $q_1 = (x_1, y_1, z_1) \in \partial A = J^+(q_0) \setminus I^+(q_0)$, $q_1 \neq q_0$.  Then an optimal trajectory in the full  \sL problem \eq{prf1}--\eq{prf4} is a future directed lightlike piecewise smooth trajectory with one or two subarcs generated by the vector fields $X_1 \pm X_2$. In detail, up to a reparametrization:
\begin{itemize}
\item[$(1)$]
If $z_1 = 0$, then 
$$
u(t) \equiv \const = (1, \pm 1), \qquad q(t) = e^{t(X_1 \pm X_2)} = (t, \pm t, 0), \qquad t \in [0, t_1], \quad t_1 = x_1. 
$$
\item[$(2)$]
If $z_1 > 0$, then 
\begin{align*}
&t \in [0, \tau_1] \then u(t) \equiv (1, - 1), \qquad q(t) = e^{t(X_1 - X_2)} = (t, -t, 0), \\
&t \in [\tau_1, \tau_1+ \tau_2] \then u(t) \equiv (1,  1), \\
&\qquad\qquad\qquad q(t) = e^{(t-\tau_1)(X_1 + X_2)}e^{\tau_1(X_1 - X_2)} = (t, t - 2\tau_1, \tau_1(t-\tau_1)), \\
&\tau_1 = \frac{x_1-y_1}{2}, \qquad \tau_2 = \frac{x_1+y_1}{2}.
\end{align*}
\item[$(3)$]
If $z_1 < 0$, then 
\begin{align*}
&t \in [0, \tau_1] \then u(t) \equiv (1,  1), \qquad q(t) = e^{t(X_1 + X_2)} = (t, t, 0), \\
&t \in [\tau_1, \tau_1+ \tau_2] \then u(t) \equiv (1,  -1), \\
&\qquad\qquad\qquad q(t) = e^{(t-\tau_1)(X_1 - X_2)}e^{\tau_1(X_1 + X_2)} = (t, 2\tau_1 -  t, -\tau_1(t-\tau_1)), \\
&\tau_1 = \frac{x_1+y_1}{2}, \qquad \tau_2 = \frac{x_1-y_1}{2}.
\end{align*}
\end{itemize}
\end{theorem}

The broken lightlike trajectories with two arcs described in items (1), (2) of Th. \ref{th:optJ} are shown in Fig. \ref{fig:S0opt}.

\begin{proof} 
Let $q(t)$, $t \in [0, t_1]$, be a future directed nonspacelike trajectory connecting $q_0$ and $q_1$. If $q(\cdot)$ is not lightlike, then there exists a future directed timelike arc $q(t)$, $t \in [s_1, s_2]$, $0 \leq s_1 < s_2 \leq t_1$, thus $q(t_1) \in \intt \A$, a contradiction. Thus $q(\cdot)$ is lightlike, and the statement follows by direct computation of trajectories of the lightlike vector fields $X_1 \pm X_2$. 
\end{proof}

\begin{corollary}
For any $q_1 \in J^+(q_0)$, $q_1 \neq q_0$, there is a unique, up to reparametri\-za\-tion, \sL length minimizer in the full problem \eq{prf1}--\eq{prf4} that connects $q_0$ and $q_1$:
\begin{itemize}
\item
if $q_1 \in \intt \A =   I^+(q_0)$, then $q(\cdot)$ is a future directed timelike strictly normal trajectory described in Theorems  $\ref{th:optI}$, $\ref{th:distIf}$.
\item
if $q_1 \in \partial \A = J^+(q) \setminus I^+(q_0)$, then $q(\cdot)$ is a future directed lightlike nonstrictly normal trajectory described in Th.  $\ref{th:optJ}$.
\end{itemize}
\end{corollary}

\begin{corollary}
Any \sL length maximizer of problem \eq{prf1}--\eq{prf4} of positive length is timelike and strictly normal.
\end{corollary}

\begin{remark}
The broken trajectories described in items $(2)$, $(3)$ of Th. {\em\ref{th:optJ}} are optimal in the \sL problem, while in sub-Riemannian problems trajectories with angle points cannot be optimal, see {\em \cite{hak_ledon}}.
Moreover, these broken trajectories are normal and nonsmooth, which is also impossible in sub-Riemannian geometry.
\end{remark}

\section{\SL distance}
Denote $d(q) := d(q_0, q)$, $q \in J^+(q_0)$.
 
\begin{theorem}\label{th:dist}
Let  $q = (x, y, z) \in J^+(q_0)$. Then
\be{dq}
d(q) = \sqrt{x^2-y^2} \cdot \frac{p}{\sinh p}, \qquad p = \b\left(\frac{z}{x^2-y^2}\right). 
\ee
In particular:
\begin{itemize}
\item[$(1)$]
$z = 0 \iff d(q) = \sqrt{x^2 - y^2}$,\\
\item[$(2)$]
$q \in J^+(q_0) \setminus I^+(q_0) \iff d(q) = 0$.
\end{itemize}
\end{theorem}

\begin{remark}\label{rem:psinhp}
In the right-hand side of the first equality in  \eq{dq}, we assume by continuity that $\frac{p}{\sinh p} = 1$ for $p = 0$ and $\frac{p}{\sinh p} = 0$ for $p = \infty$. See the plot of the function $\frac{p}{\sinh p} $ in Fig. $\ref{fig:pSinhp}$.

\figout{
\onefiglabelsizen
{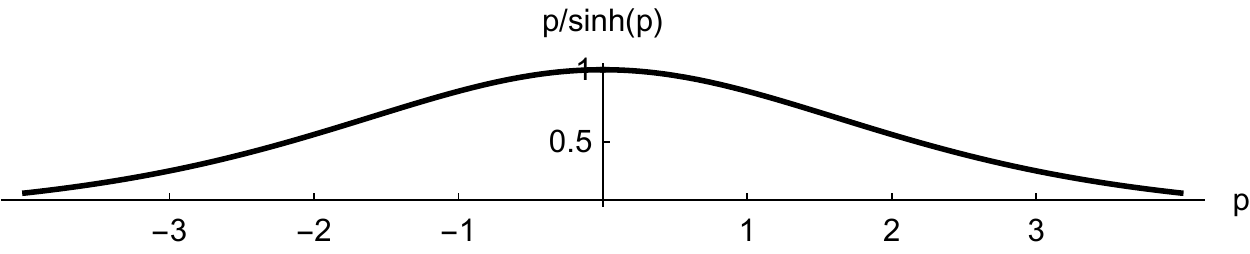}{Plot of $\frac{p}{\sinh p} $}{fig:pSinhp}{2}
}
\end{remark}

\begin{proof}
Let $q \in I^+(q_0)$, then the \sL length maximizers from $q_0$ to $q$ are described in Theorem   \ref{th:distIf} and the expression for $\restr{d}{\tA} = \td$  was obtained in Th. \ref{th:distI}. In particular, if $z = 0$, then $p = 0$ and $d(q) = \sqrt{x^2 - y^2}$, and vice versa.

Let $q \in J^+(q_0)\setminus I^+(q_0)$, then the \sL length maximizers from $q_0$ to $q$ are described in Th. \ref{th:optJ}. Thus $d(q) = 0$, which agrees with \eq{dq} since in this case $\frac{|z|}{x^2-y^2} = \frac 14$, so $p = \infty$.
\end{proof}

We plot restrictions of the \sL distance to several planar domains:
\begin{itemize}
\item
$\restr{d}{z=0}= \sqrt{x^2 - y^2}$ to the domain $J^+(q_0) \cap \{ z = 0 \} = \{ x \geq |y|, \ z = 0\}$, see Fig. \ref{fig:dz=0},
\item
$\restr{d}{y=0}$ to the domain $J^+(q_0) \cap \{ y = 0 \} = \{-x^2/4 \leq z \leq x^2/4, \ y = 0\}$, see Fig. \ref{fig:dy=0},
\item
$\restr{d}{x=1}$ to the domain $J^+(q_0) \cap \{ x = 1 \} = \{y^2 + 4 |z|\leq 1, \ x = 1\}$, see Fig. \ref{fig:dx=1}.
\end{itemize}

\figout{
\twofiglabelsizeh
{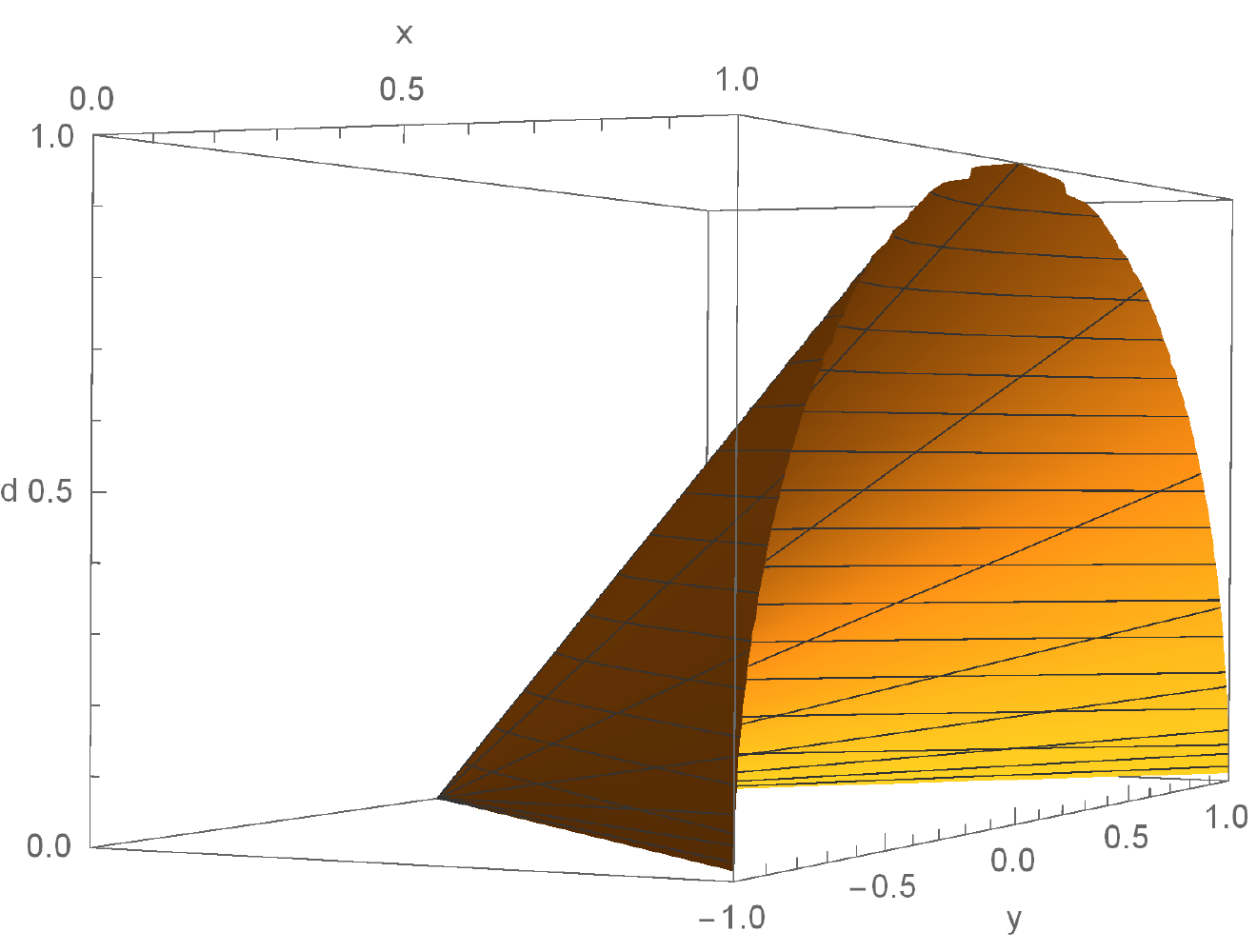}{Plot of $\restr{d}{z=0}$}{fig:dz=0}{6}
{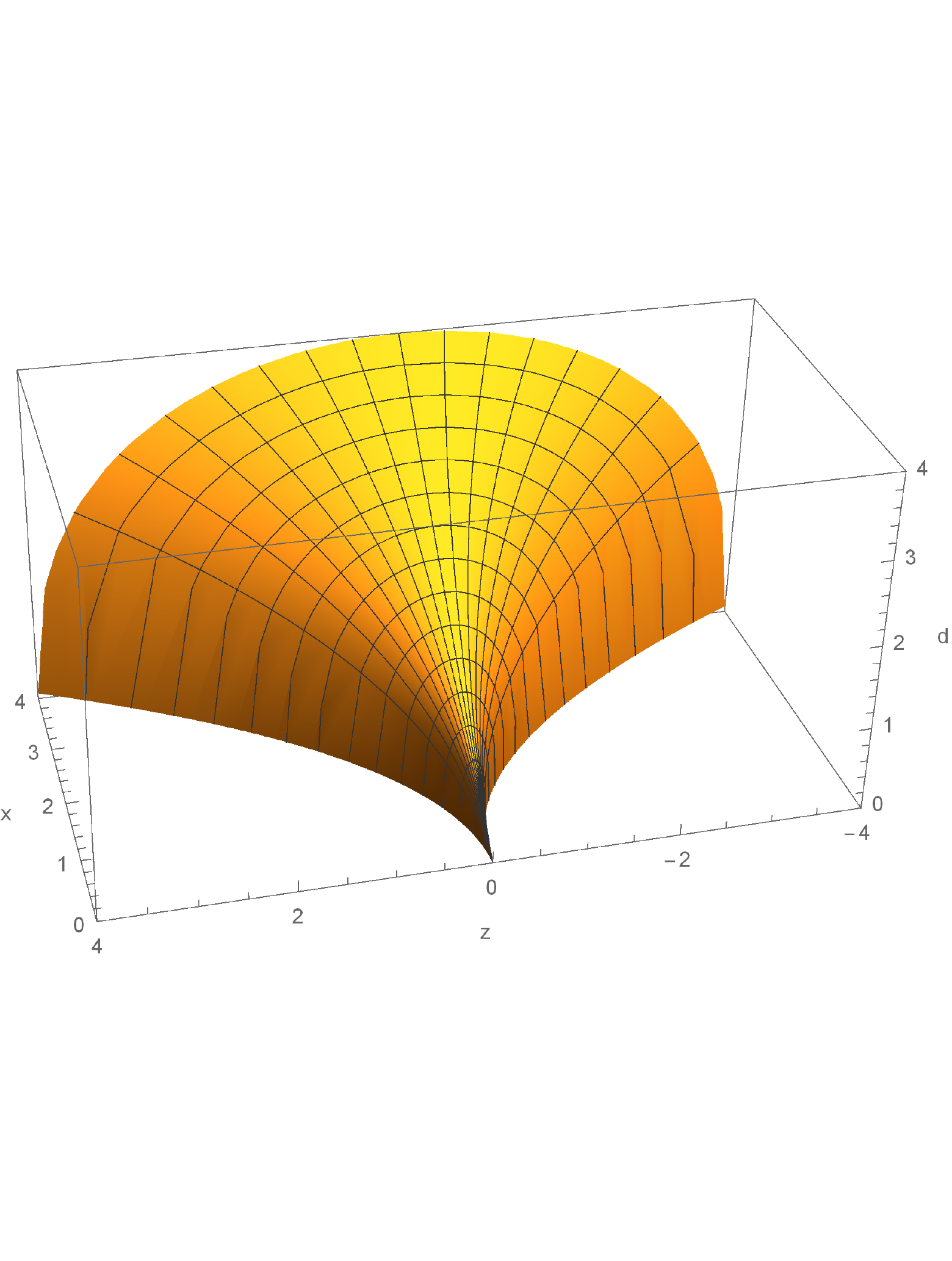}{Plot of $\restr{d}{y=0}$}{fig:dy=0}{10}

\onefiglabelsizen
{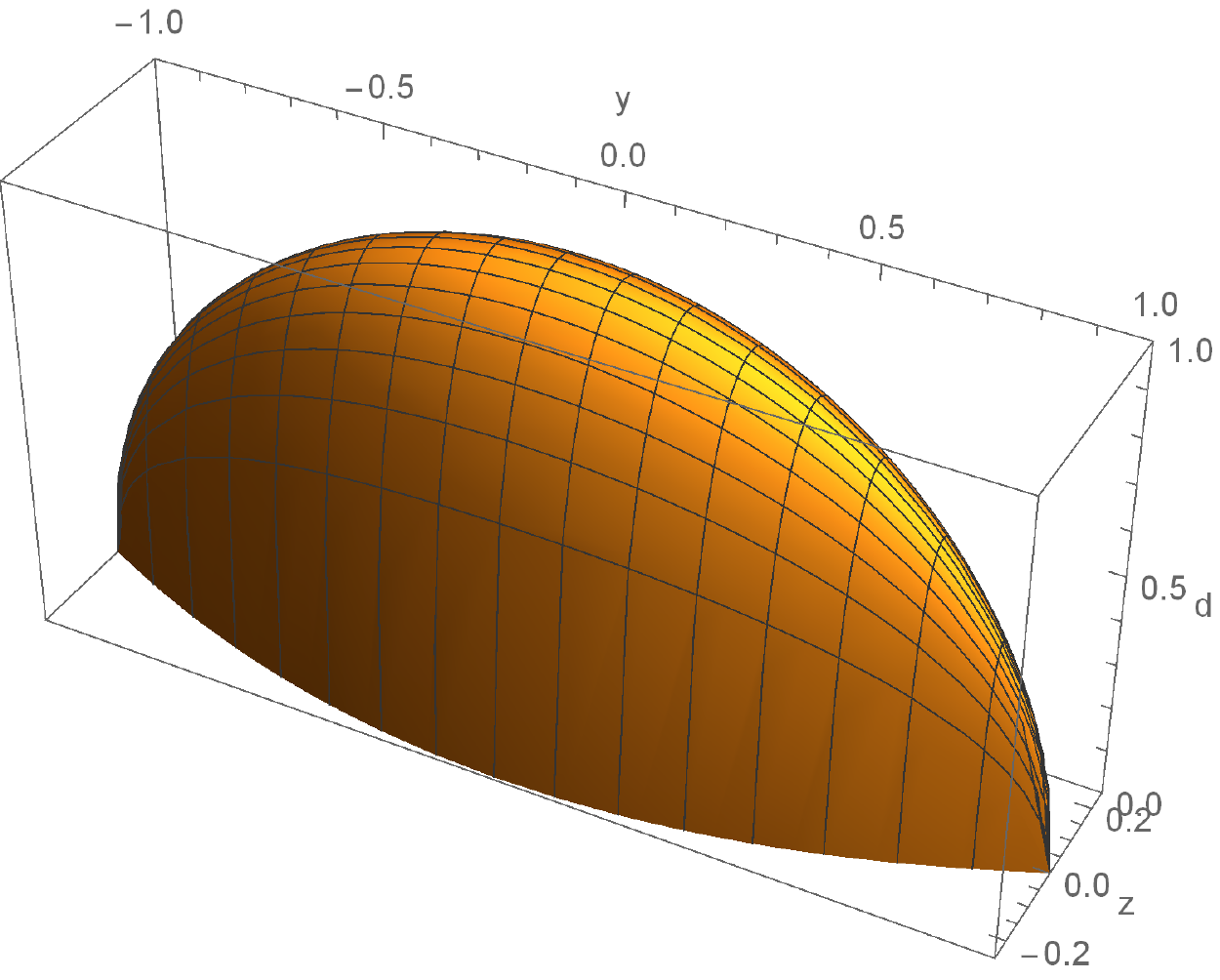}{Plot of $\restr{d}{x=1}$}{fig:dx=1}{7}
}

The \sL distance has the following regularity properties.

\begin{theorem}\label{th:dreg}
\begin{itemize}
\item[$(1)$]
The function $d(\cdot)$ is continuous on $J^+(q_0)$ and real-analytic on $I^+(q_0)$.
\item[$(2)$]
The function $d(\cdot)$ is not Lipschitz near points $q = (x, y, z)$ with $x = |y|>0$, $z = 0$.
\end{itemize}
\end{theorem}
\begin{proof}
(1) 
 follows from representation \eq{dq}.

\medskip
(2) follows from item (1) of Th. \ref{th:dist} since the function $\restr{d}{z = 0} =\sqrt{x^2-y^2}$ is not Lipschitz near points  with $x = |y|>0$.
\end{proof}

\begin{remark}
Item $(1)$ of Th. $\ref{th:dreg}$ improves item $(2)$ of Sec. $\ref{sec:groch}$.
\end{remark}

\begin{remark}
Item $(2)$ of Th. $\ref{th:dreg}$ is visualized in Fig. $\ref{fig:dz=0}$ since the cone given by the plot of $\restr{d}{z = 0} = \sqrt{x^2-y^2}$ has vertical tangent planes at points $x = |y|> 0$.

Moreover, item $(2)$ of Th. $\ref{th:dreg}$ can be essentially detailed by a precise description of the asymptotics of the \sL distance $d(q)$ as $q \to \partial \A$, this will be done in a forthcoming paper {\em \cite{pop_sach}}.
\end{remark}

\begin{remark}
The \sL distance $\map{d}{J^+(q_0)}{[0, + \infty)}$ is not uniformly continuous since the same holds for its restriction $\restr{d}{z = 0} = \sqrt{x^2-y^2}$ on the angle $\{x \geq |y|\}$.
\end{remark}

As was shown in \cite{groch6}, the \sL distance $d(q)$  admits a lower bound by the function $\sqrt{x^2 - y^2 - 4 |z|}$ and does not admit an upper bound by this function multiplied by any constant (see item (4) in Sec. \ref{sec:groch}).   Here we precise this statement and prove another upper bound.

\begin{theorem}
\begin{itemize}
\item[$(1)$]
The ratio $\dfrac{\sqrt{x^2 - y^2 - 4 |z|}}{d(q)}$ takes any values in the segment $[0, 1]$ for $q =(x,y,z) \in J^+(q_0)$. 
\item[$(2)$]
For any  $q = (x, y, z) \in J^+(q_0)$ there holds the bound $d(q) \leq \sqrt{x^2-y^2}$,
moreover, the ratio $\dfrac{d(q)}{\sqrt{x^2-y^2}}$ takes any values in the segment $[0, 1]$.
\end{itemize}
\end{theorem}

The two-sided bound
\be{bound}
{\sqrt{x^2 - y^2 - 4 |z|}} \leq {d(q)} \leq \sqrt{x^2-y^2}, \qquad q \in J^+(q_0),
\ee
is visualized in Fig. \ref{fig:dbound}, which shows plots of the surfaces (from below to top):
$$
\sqrt{x^2-y^2} = 1, \qquad  {d(q)} = 1, \qquad    {\sqrt{x^2 - y^2 - 4 |z|}} = 1, \qquad q \in J^+(q_0).
$$

\figout{
\onefiglabelsizen{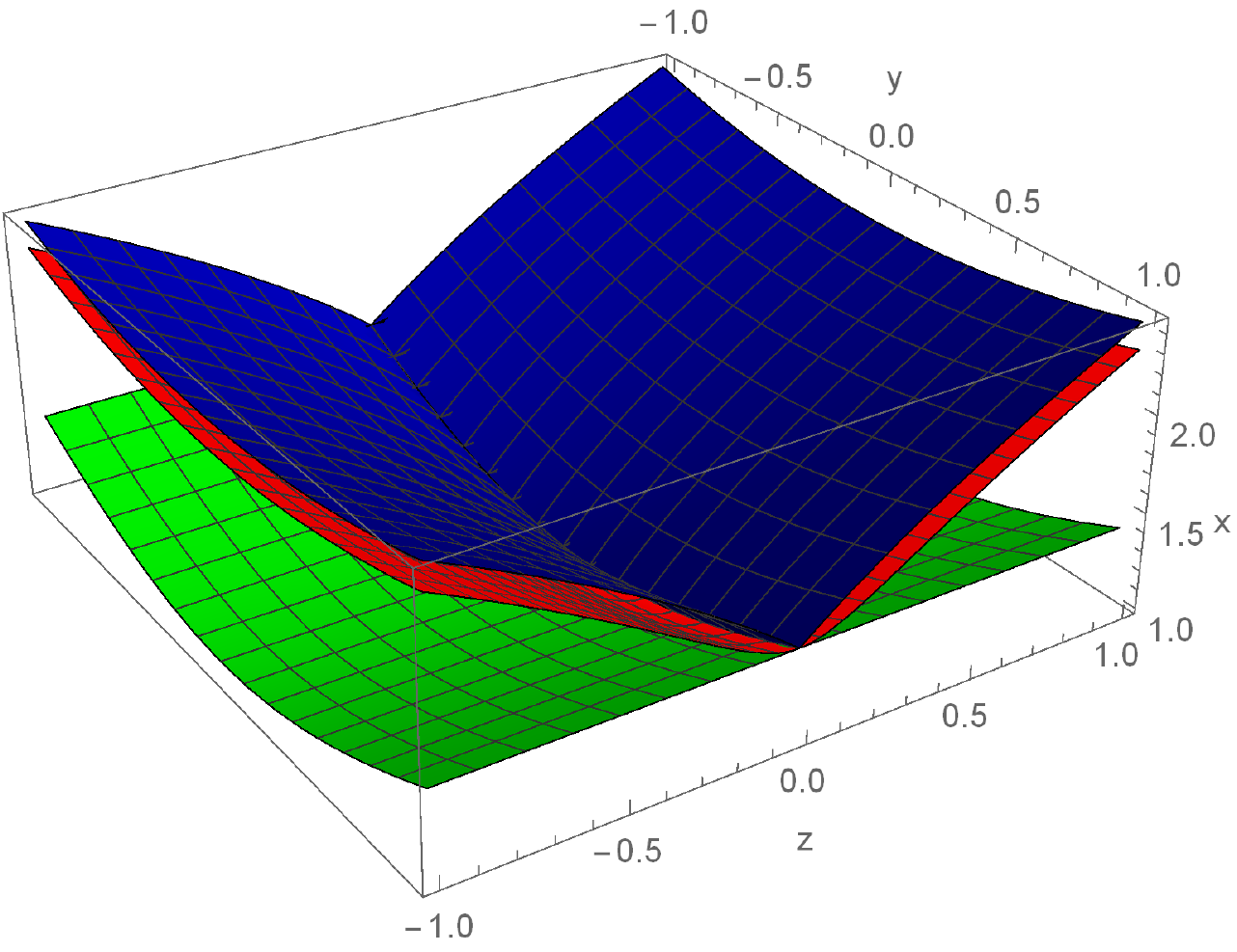}{Bound \eq{bound}}{fig:dbound}{6}
}

\begin{proof}
$(1)$
It follows from \eq{dq} that
$$
\frac{x^2 - y^2 - 4 |z|}{d^2(q)} = \frac{\sinh^2 p - \sinh p \cosh p + p}{p^2},
$$
and the function in the right-hand side takes all values in the segment $[0, 1]$ for $q \in J^+(q_0)$.

\medskip
$(2)$
It follows from \eq{dq} that $\frac{d(q)}{\sqrt{x^2-y^2}} = \frac{p}{\sinh p}$. When $q \in J^+(q_0)$, the ratio $\frac{p}{\sinh p}$ takes all values in the segment $[0, 1]$, see Remark \ref{rem:psinhp} after Th. \ref{th:dist}.
\end{proof}

\section{Symmetries}
\begin{theorem}\label{th:sym}
\begin{itemize}
\item[$(1)$]
The hyperbolic rotations $X_0 = y \pder{}{x} + x \pder{}{y}$ and reflections
$\eps^1 \ : \ (x, y, z) \mapsto (x, - y, z)$, $\eps^2 \ : \ (x, y, z) \mapsto (x, y, -z)$  preserve  $d(\cdot)$.
\item[$(2)$]
The dilations $Y = x \pder{}{x} + y \pder{}{y} + 2z \pder{}{z}$ stretch $d(\cdot)$: 
$$
d(e^{sY}(q)) = e^s d(q), \qquad s \in \R, \quad q \in J^+(q_0).
$$
\end{itemize} 
\end{theorem}
\begin{proof}
(1) The flow of the hyperbolic rotations
$$
e^{s X_0} \ : \ (x, y, z) \mapsto (x \cosh s + y \sinh s, x \sinh s + y \cosh s, z), \qquad
 s \in \R, \quad (x,y,z) \in M,
$$
preserves the exponential mapping:
$$
e^{sX_0} \circ \Exp(\psi, c, t) = \Exp(\psi + s, c, t), \qquad (\psi, c, t) \in N, \quad s \in \R,
$$
thus $d(e^{sX_0}(q)) = d(q)$ for $q \in I^+(q_0)$. Moreover, the flow $e^{sX_0}$ preserves the boundary $\partial \A = J^+(q_0) \setminus I^+(q_0)$, thus $d(e^{sX_0}(q)) = d(q) = 0$ for $q \in J^+(q_0) \setminus I^+(q_0)$.

Further, it is obvious from \eq{dq} that the reflections $\eps^1$, $\eps^2$ preserve $d(\cdot)$.

\medskip

(2) The flow of the dilations
$$
e^{s Y} \ : \ (x, y, z) \mapsto (x e^s, ye^s, ze^{2s}), \qquad
 s \in \R, \quad (x,y,z) \in M,
$$
acts on  the exponential mapping as follows:
$$
e^{sY} \circ \Exp(\psi, c, t) = \Exp(\psi, ce^{-2s}, te^s), \qquad (\psi, c, t) \in N, \quad s \in \R,
$$
thus $d(e^{sY}(q)) = e^s d(q)$ for $q \in I^+(q_0)$. The equality $d(e^{sY}(q)) = e^s d(q) = 0$ for $q \in J^+(q_0) \setminus I^+(q_0)$ follows since the flow $e^{sY}$ preserves the boundary $\partial \A = J^+(q_0) \setminus I^+(q_0)$.
\end{proof}

\section{\SL spheres}
\subsection{Spheres of positive radius}
\SL spheres 
$$
S(R) = \{ q \in M \mid d(q) = R\}, \qquad R> 0,
$$
are transformed one into another by dilations:
$$
S(e^s R) = e^{sY}(S(R)), \qquad s \in \R,
$$
thus we describe the unit sphere 
\be{S1}
S = S(1) = \{\Exp(\lam, 1) \mid \lam \in C\}.
\ee

\begin{theorem}
\begin{itemize}
\item[$(1)$]
The unit \sL sphere  $S$  is a regular real-analytic manifold diffeomorphic to   $\R^2$.
\item[$(2)$]
Let  $q = \Exp(\psi, c, 1) \in S$, $(\psi, c) \in C$,  then the tangent space
\be{TqS}
T_qS = \left\{v = \sum_{i=1}^3 v_i X_i(q) \mid - v_1 \cosh(\psi + c)+v_2 \sinh(\psi+c)+v_3c=0\right\}.
\ee
\item[$(3)$]
$S$  is the graph of the function $x = \sqrt{y^2 + f(z)}$, where $f(z) = e \circ k(z)$, $e(w) = \frac{\sinh^2 w}{w^2}$, $k(z) = b(z)/2$, $b = a^{-1}$, $a(c) = \frac{\sinh c - c}{2c^2}$.
\item[$(4)$]
The function $f(z)$  is real-analytic, even, strictly convex, unboundedly and strictly increasing for   $z \geq 0$. This function has a Taylor decomposition   $f(z) = 1 + 12 z^2 + O(z^4)$ as   $z \to 0$ and 
an asymptote $4|z|$ as $z \to \infty$:
\be{fas}
\lim_{z \to \infty} (f(z) - 4 |z|) = 0.
\ee
\item[$(5)$]
The function $f(z)$ satisfies the bounds
\be{fbound}
4 |z| < f(z) < 4 |z| + 1, \qquad z \neq 0. 
\ee
\item[$(6)$]
A section of  the sphere $S$ by  a plane $\{z = \const\}$  is a branch of the hyperbola $x^2-y^2 = f(z)$, $x>0$.
A section of  the sphere   $S$ by  a plane  $\{x = \const>1 \}$  is a strictly convex curve   $y^2+f(z) = x^2$ diffeomorphic to  $S^1$.
\item[$(7)$]
The \sL distance from the point $q_0$ to a point 
  $q = (x, y, z) \in \tA$ may be expressed as  
$d(q) = R$, where $x^2-y^2 = R^2 f(z/R^2)$.
\item[$(8)$]
The \sL ball $B = \{ q \in M \mid d(q) \leq 1\}$  has infinite volume in the coordinates   $x, y, z$.
\end{itemize} 
\end{theorem}

See in Fig. \ref{fig:plotS} a plot of the sphere $S$ (above in red) and the Heisenberg beak $\partial \A$ (at the bottom in blue). Different \sL length maximizers connecting $q_0$ and $S$ are shown in Fig. \ref{fig:Sopt}. A plot of the function $f(z)$ illustrating bound \eq{fbound} is shown in Fig. \ref{fig:f}. Sections of the sphere $S$ by the planes $\{x = 1, 2, 3\}$ are shown in Fig. \ref{fig:Sx}.

\figout{
\twofiglabelsizeh
{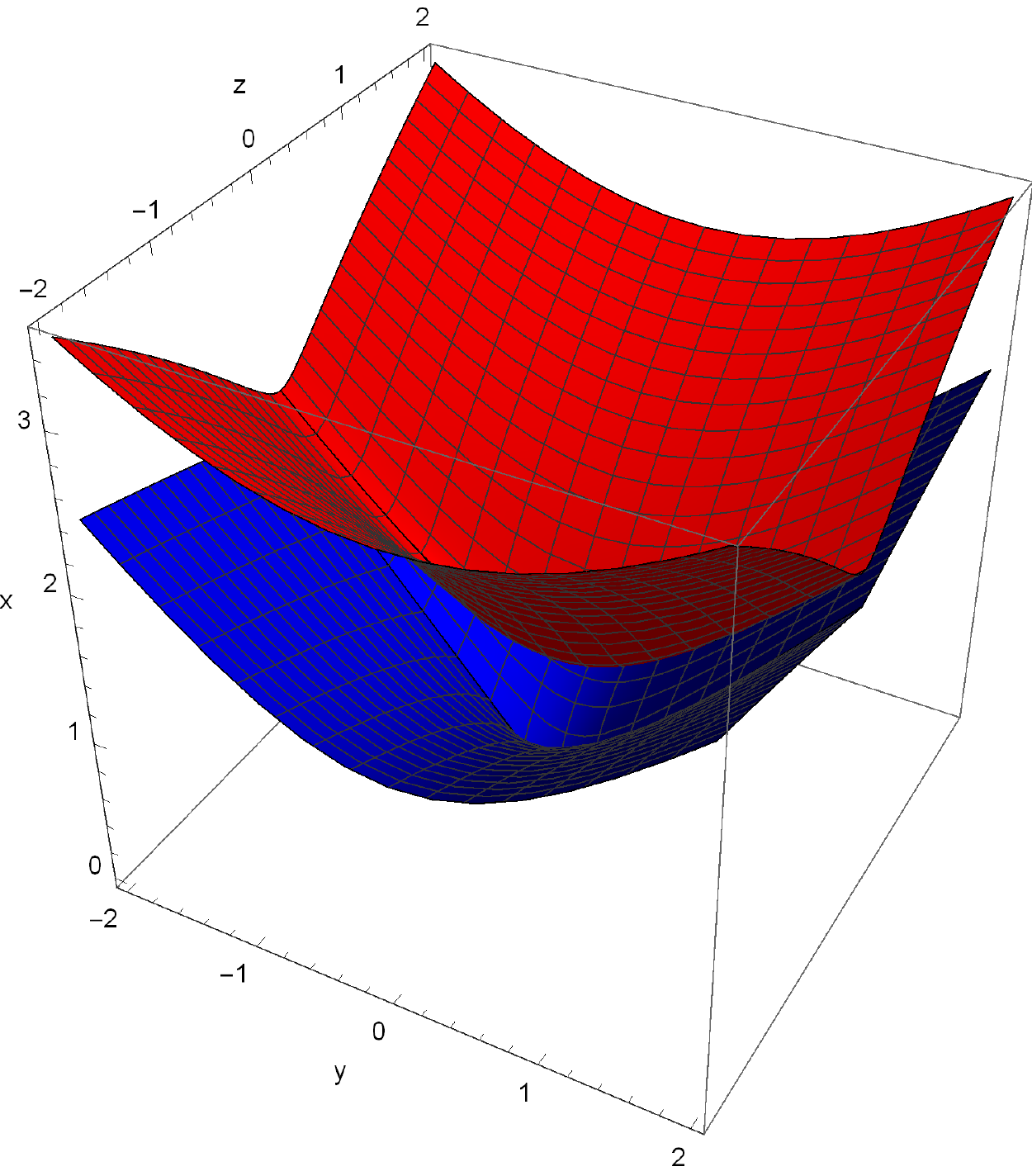}{The sphere $S$ and the Heisenberg beak $\partial \A$}{fig:plotS}{6}
{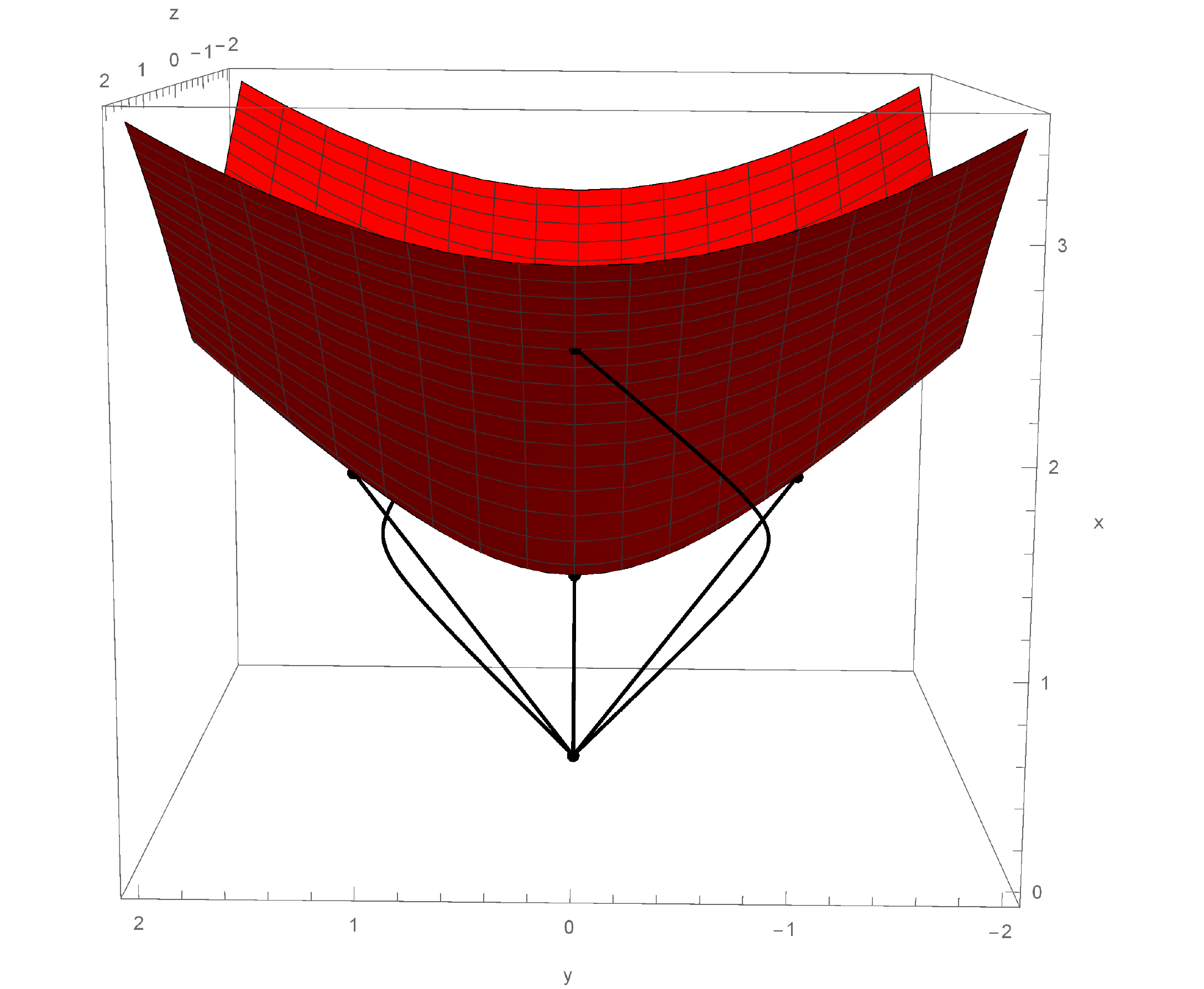}{Maximizers connecting $q_0$ and $S$}{fig:Sopt}{6}

\twofiglabelsizeh
{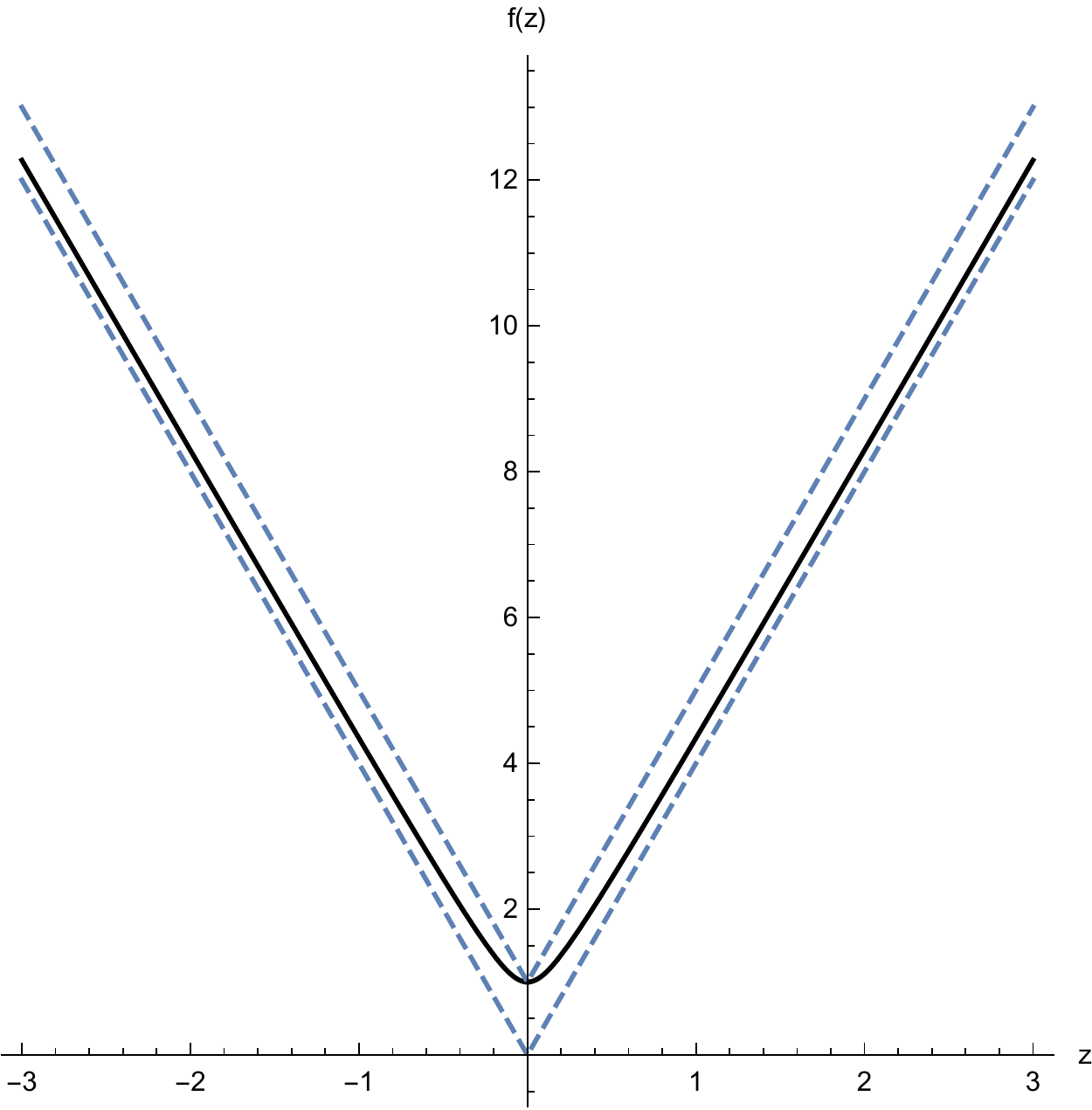}{Plot of $f(z)$ and bound \eq{fbound}}{fig:f}{5}
{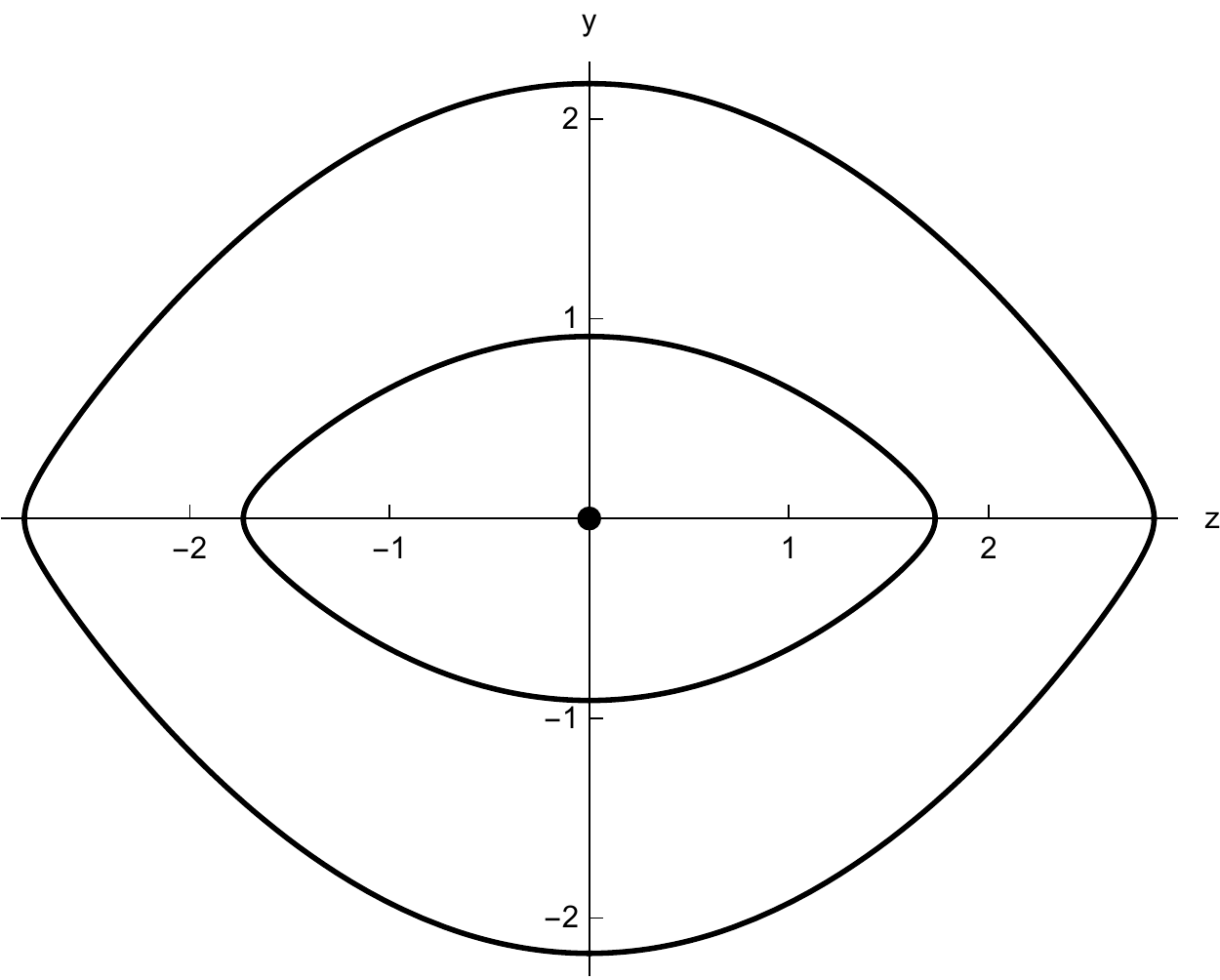}{Sections of $S$ by the planes $\{x = 1, 2, 3\}$}{fig:Sx}{5}
}

\begin{proof}
$(1)$
Since $\map{\Exp}{C \times \R_+}{\tA}$ is a diffeomorphism, the parametrization \eq{S1} of the sphere $S$ implies that it is a smooth 2-dimensional manifold diffeomorphic to $\R^2$. Moreover, the exponential mapping is real-analytic, thus $S$ is real-analytic as well.

\medskip
$(2)$
Let
$q = \Exp(\lam_0, 1) \in S$, $\lam_0 = (\psi, c, q_0 ) \in C$, and let $\lam_1 = e^{\vec H}(\lam_0)$. Then
\be{TqS1}
T_qS = \lam_1^{\perp} = \{v \in T_q M \mid \langle \lam_1, v \rangle = 0\}.
\ee
Since $h_1(\lam_1) = - \cosh(\psi+c)$, $h_2(\lam_1) = \sinh(\psi+c)$, $h_3(\lam_1) = c$, representation \eq{TqS} follows from \eq{TqS1}.

\medskip
$(3)$
It follows from \eq{TqS} that the 2-dimensional manifold $S$ projects regularly to the coordinate plane $(y, z)$, thus it is a graph of a real-analytic function $x = F(y, z)$. Since $e^{tX_0}(S) = S$, $t \in \R$, then 
$$
0 = \restr{X_0(F(y,z)-x)}{S} = F(y, z) \pder{F}{y}(y,z) - y.
$$ 
Integrating this differential equation, we get $F(y, z) = \sqrt{y^2 + f(z)}$ for a real-analytic function $f(z)$.

Since $S \cap \{ z = 0\} = \left\{x =  \sqrt{y^2 + 1}, \  z = 0\right\}$, then $f(0) = 1$.

Let $z \neq 0$. Then $z = \frac{\sinh c - c}{2c^2} = a(c)$ by virtue of \eq{qcn0}. The function $\map{a}{\R}{\R}$ is a diffeomorphism, denote the inverse function $b = a^{-1}$. 
By virtue of \eq{x2-y2}, we have
$f(z) = x^2 - y^2 = \frac{4}{c^2} \sinh^2 p$, whence $f(a(c)) = \frac{4}{c^2} \sinh^2 p$, thus $f(a) = e(\frac b2(a))$, where $e(x) = \frac{\sinh^2 x}{x^2}$. Item (3) follows.

\medskip
$(4)$
We have already proved that $f(z)$ is real-analytic. Since $\eps^1(S) = S$, then $f$ is even. Immediate computation shows that $k'(z) > 0$, $z > 0$, and $e'(x)> 0$, $x > 0$, whence $f'(z) > 0$, $z > 0$. Similarly it follows that $f''(z) > 0$ for $z > 0$. By virtue of the expansions $k(z) = 6 z + O(z^2)$, $z \to 0$ and $e(x) = 1 + \frac{x^2}{3} + O(x^4)$, $x \to 0$, we get $f(z) = 1 + 12 z^2 + O(z^4)$, $z \to 0$. Finally, it easily follows from the definition of the function $f(z)$ that $\lim_{z \to \infty}(f(z) - 4 |z|) = 0$.

\medskip
$(5)$
follows from (4).

\medskip
$(6)$
It is straightforward that $S \cap \{z = \const\} = \{x^2 - y^2 = f(z), \ x > 0, \ z = \const\}$ is a branch of a hyperbola. 

The section $S \cap \{x = \const > 1\} = \{y^2 + f(z) = x^2, \ x = \const > 1\}$ is a smooth compact curve, thus diffeomorphic to $S^1$. If $y \geq 0$, then this curve is given by the equation $y = \sqrt{x^2 - f(z)}$, which is a strictly concave function (this follows by twice differentiation).

\medskip
$(7)$
Take any point $q = (x,y,z) \in \tA$, then there exists $s \in \R$ such that $e^{-s Y} (q) \in S$, i.e., $d(q) = e^s$, see item (2) of Th. \ref{th:sym}. Denoting $R = e^s$, we get $\frac xR = \sqrt{\frac{y^2}{R^2} + f\left(\frac{z}{R^2}\right)}$, and item (7) of this theorem follows.

\medskip
$(8)$
The unit ball is given explicitly by
$$
B = \left\{(x, y, z) \in \R^3 \mid \sqrt{y^2 + 4 |z|} \leq x \leq \sqrt{y^2 + f(z)}\right\},
$$ 
thus its volume is evaluated by the integral
$$
V(B) = \int_{-\infty}^{+ \infty} dy \int_{-\infty}^{+ \infty} dz \left( \sqrt{y^2 + f(z)} - \sqrt{y^2 + 4 |z|}\right) = + \infty.
$$
\end{proof}

\begin{remark}
Thanks to bound \eq{fbound} of the function $f(z)$, the sphere $S = \left\{ x = \sqrt{y^2 + f(z)}\right\}$ is contained in the domain
$$
\left\{q = (x, y, z) \in M \mid \sqrt{y^2 + 4 |z|} < x \leq \sqrt{y^2+ 4 |z| + 1}\right\}.
$$ 
The bounding functions of this domain provide an approximation of the function $\sqrt{y^2 + f(z)}$ defining  $S$ up to the accuracy 
$$
\sqrt{y^2+ 4 |z| + 1} - \sqrt{y^2 + 4 |z|} 
= 
\frac{1}{\sqrt{y^2+ 4 |z| + 1} + \sqrt{y^2 + 4 |z|} }  
\leq \min\left(1, \frac{2}{|y|}, \frac{1}{\sqrt{|z|}}\right).
$$
\end{remark}

\subsection{Sphere of zero radius}
Now consider the zero radius sphere
$$
S(0) = \{q \in M \mid d(q) = 0\}.
$$

\begin{theorem}\label{th:S0}
\begin{itemize}
\item[$(1)$]
$S(0) = J^+(q_0) \setminus I^+(q_0) = \partial J^+(q_0) =  \partial I^+(q_0) = \partial \A$.
\item[$(2)$]
$S(0)$ is the graph of a continuous function $x = \Phi(y, z) := \sqrt{y^2 + 4 |z|}$, thus a $2$-dimensional topological manifold.
\item[$(3)$]
The function $\Phi(y, z)$ is even in $y$ and $z$, real-analytic for $z \neq 0$, Lipschitz near $z = 0$, $y \neq 0$, and H\"older with constant $\frac 12$, non-Lipschitz near $(y, z)= (0, 0)$. 
\item[$(4)$]
$S(0)$ is filled by broken lightlike trajectories with one or two edges described in  Th. $\ref{th:optJ}$, and is parametrized by them as follows:
\begin{multline*}
S(0) = \left\{ e^{\tau_2 (X_1-X_2)} e^{\tau_1(X_1+X_2)} = (\tau_1 + \tau_2, \tau_1-\tau_2, -\tau_1 \tau_2) \mid \tau_i \geq 0 \right\} \\
\cup
\left\{ e^{\tau_2 (X_1+X_2)} e^{\tau_1(X_1-X_2)} = (\tau_1 + \tau_2, \tau_2-\tau_1, \tau_1 \tau_2) \mid \tau_i \geq 0 \right\}.
\end{multline*}
\item[$(5)$]
The flows of the vector fields $Y, X_0$ preserve $S(0)$.
 Moreover, the symmetries $Y$, $X_0$ provide a regular parametrization of 
\begin{align}
S(0) \cap \{\sgn z = \pm 1\} &= \left\{e^{sY} \circ e^{rX_0} (q_{\pm}) \mid r, s > 0\right\},  \label{S0par1} 
\end{align}
where $q_{\pm} = (x_{\pm}, y_{\pm}, z_{\pm})$ is any point in $S(0) \cap \{\sgn z = \pm 1\}$.
\item[$(6)$]
The sphere $S(0) = \left\{ 16z^2 = (x^2-y^2)^2, \ x^2 - y^2 \geq 0, \ x \geq 0\right\}$ is a semi-algebraic set. 
\item[$(7)$]
The zero-radius sphere is a Whitney stratified set with the stratification
\begin{multline*}
S(0) = \big(S(0) \cap \{z > 0\} \big) \cup \big(S(0) \cap \{z < 0\}\big) \\
\cup \big(S(0) \cap \{z = 0, \ y > 0\}\big) \cup \big(S(0) \cap \{z = 0, \ y < 0\}\big) \cup \{q_0\}.
\end{multline*} 
\item[$(8)$]
Intersection of the sphere $S(0)$ with a plane $\{z = \const \neq 0 \}$ is a branch of a hyperbola $\{x^2-y^2 = 4 |z|, \ x > 0, z = \const\}$, intersection with a plane $\{z = 0 \}$ is an angle   $\{  x = |y|, z = 0\}$, intersection with a plane $\{y = k x \}$, $k \in (-1, 1)$, is a union of two half-parabolas $\{4z = \pm(1-k^2)x^2, \ x \geq 0, \ y = kx\}$, and intersection with a plane $\{y = \pm x \}$ is a ray $\{y = \pm x, \ z = 0\}$.
\end{itemize}
\end{theorem}

The Heisenberg beak $S(0) = \partial \A$ is plotted in Figs. \ref{fig:beak}--\ref{fig:beakz} as a graph of the function $x = \sqrt{y^2 + 4|z|}$ by virtue of \eq{Jq0}, and in  Fig. \ref{fig:beak1} as a parametrized surface by virtue of \eq{S0par1}  with $q_{\pm} = (2, 0, \pm 1)$. 

\figout{
\onefiglabelsizen
{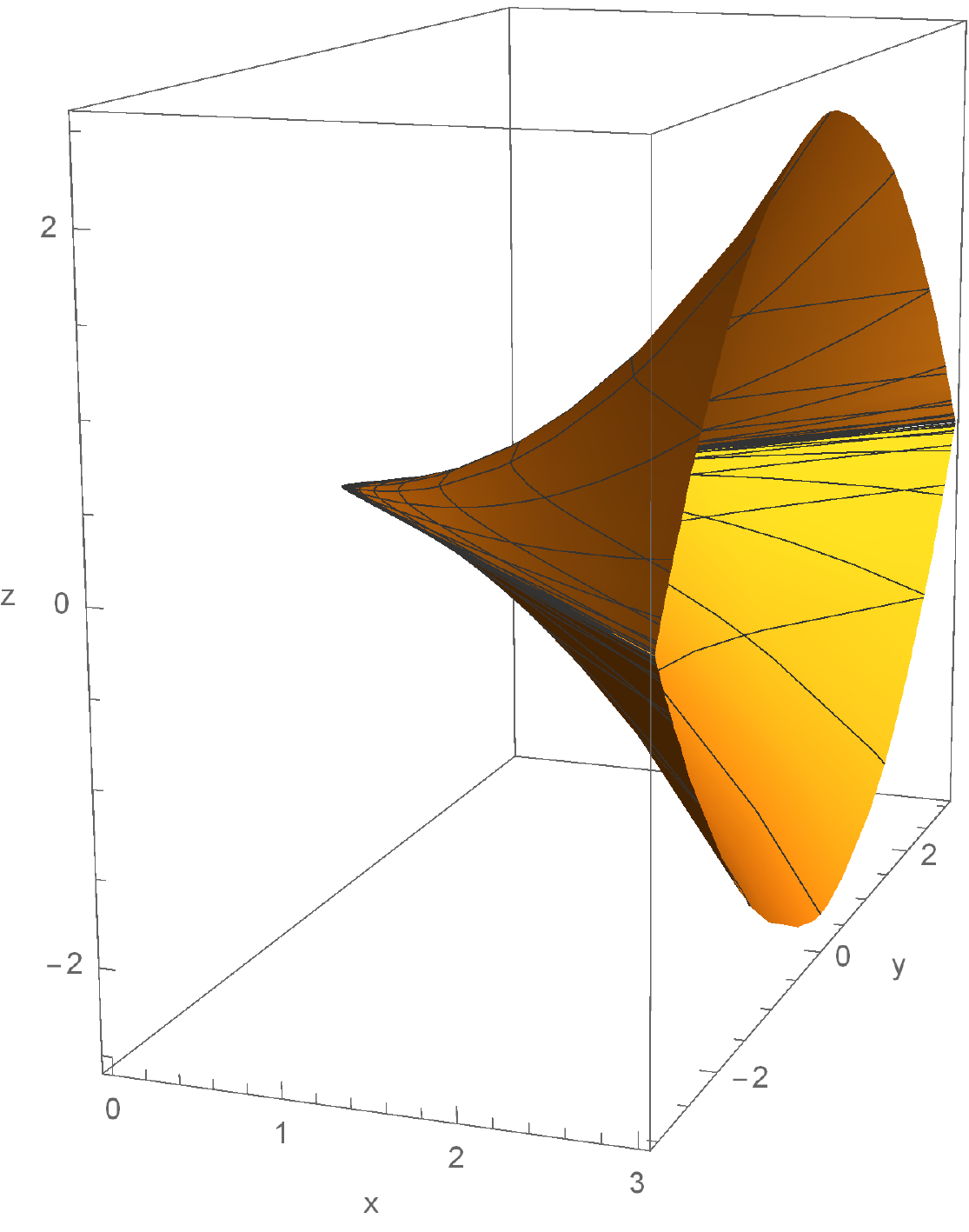}{The Heisenberg beak $\partial \A$}{fig:beak1}{8}
}

\begin{proof}
$(1)$, $(2)$
follow from item (2) of Th. \ref{th:dist} and item (3) of Sec. \ref{sec:groch}.

\medskip
$(3)$ and $(6)$--$(8)$ are obvious.

\medskip
$(4)$ follows from Th. \ref{th:optJ}.

\medskip
$(5)$ follows from Th. \ref{th:sym}.
\end{proof}

Lightlike maximizers filling $S(0)$ are shown in Fig. \ref{fig:S0opt}.
\SL spheres or radii 0, 1, 2, 3 are shown in Fig. \ref{fig:S0123}.

\figout{
\twofiglabelsizeh
{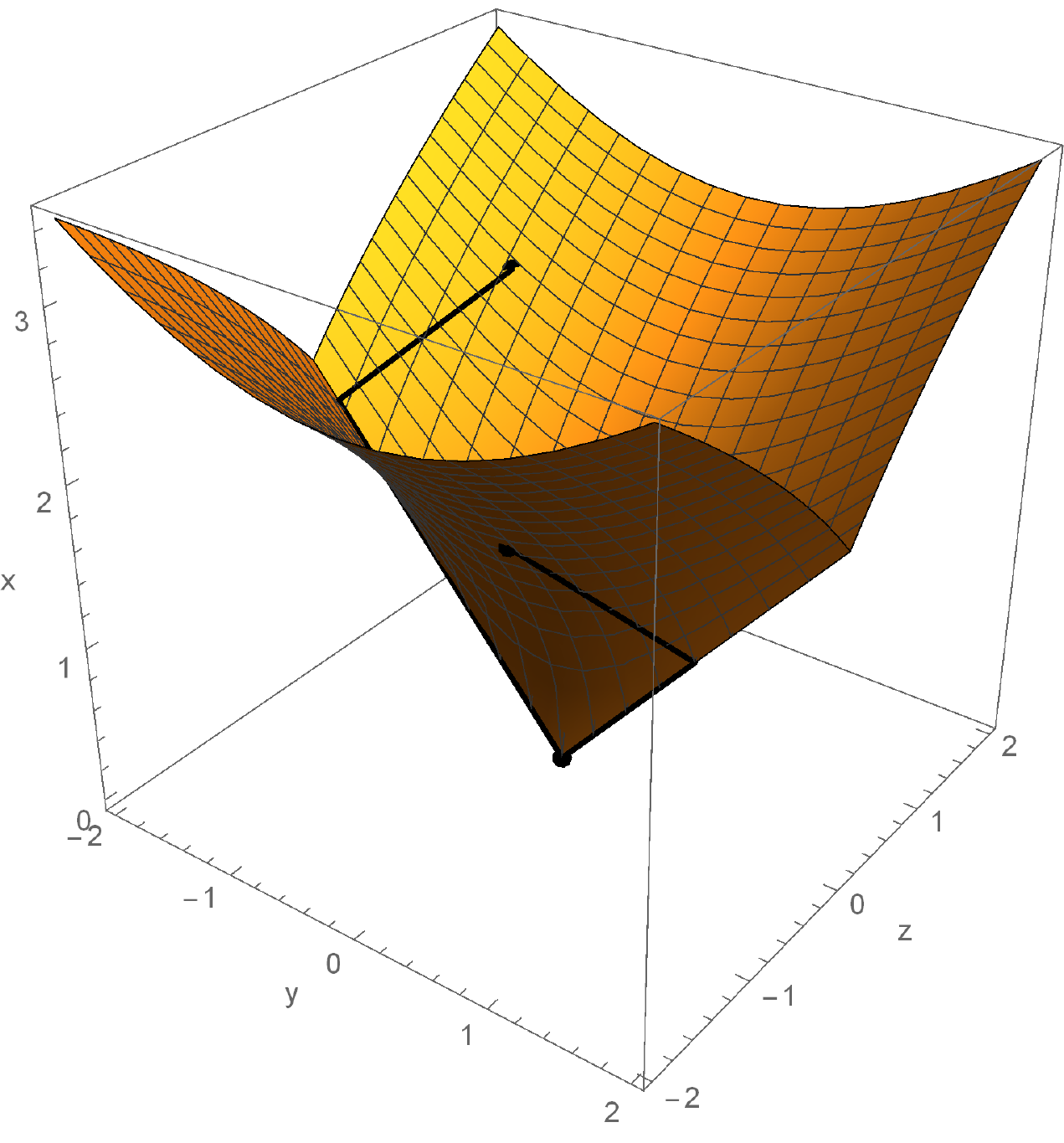}{Lightlike maximizers filling $S(0)$}{fig:S0opt}{6}
{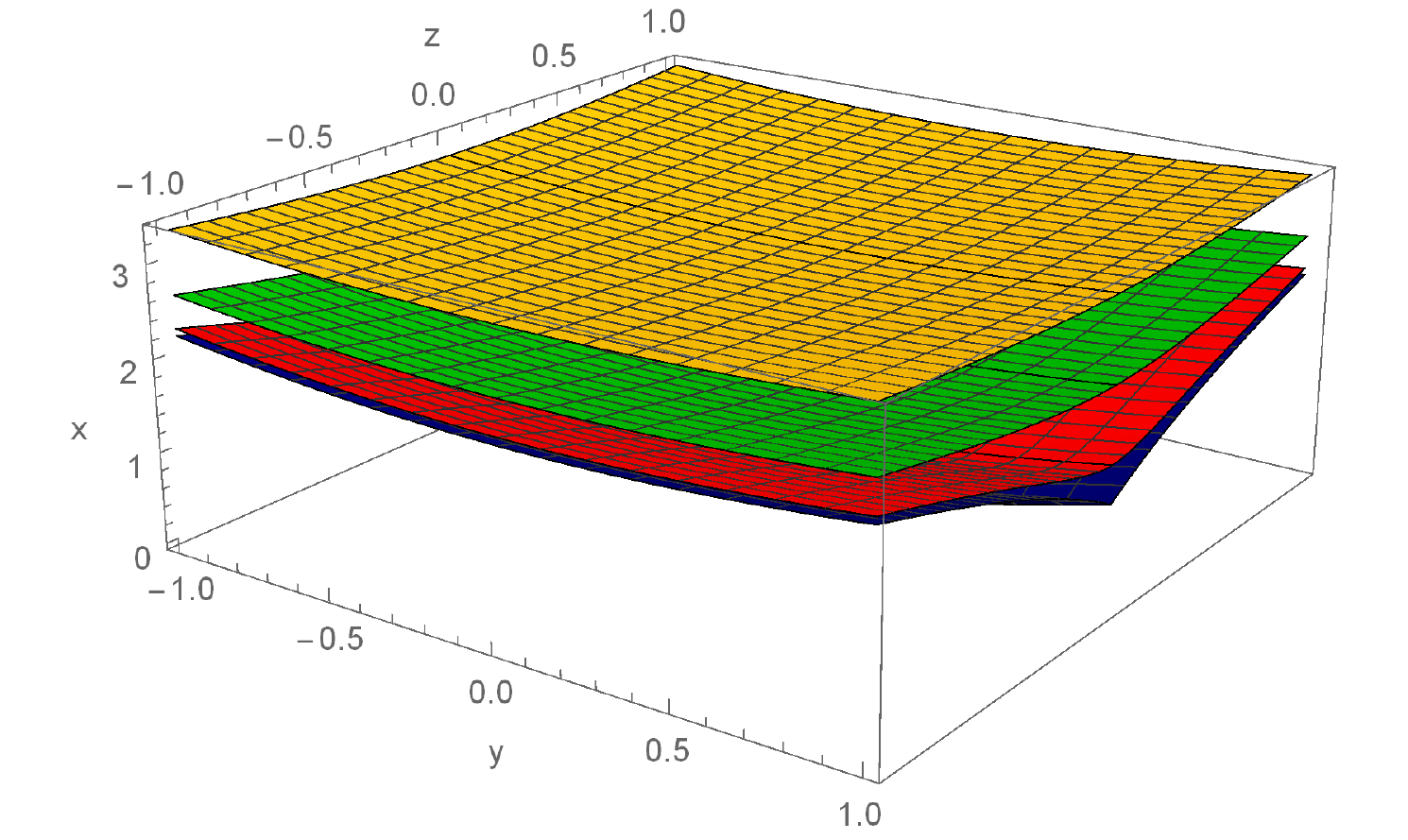}{\SL spheres or radii 0, 1, 2, 3}{fig:S0123}{5}
}

\begin{remark}
The spheres
\begin{align*}
&S(1) = \left\{(x, y, z) \in M \mid x = \sqrt{y^2 + f(z)}, \ y, z \in \R\right\}, \\
&S(0) = \left\{(x, y, z) \in M \mid x = \sqrt{y^2 + 4|z|}, \ y, z \in \R\right\}
\end{align*}
tend one to another as $z \to \infty$ since for any $y \in \R$
$$
\lim_{z \to \infty} \left(\sqrt{y^2 + f(z)} - \sqrt{y^2 + 4|z|}\right) = 0
$$
by virtue of \eq{fas}. The same holds for any spheres $S(R_1)$, $S(R_2)$, $R_i \in [0, + \infty)$. 
\end{remark}

\section{Conclusion}
The results obtained in this paper for the \sL problem on the Heisenberg group differ drastically from the known results for the sub-Riemannian problem on the same group:
\begin{enumerate}
\item
The \sL problem is not completely controllable.
\item
Filippov's existence theorem for optimal controls cannot be immediately applied to the \sL problem.
\item
In the \sL problem all  extremal trajectories are infinitely optimal, thus the cut locus and the conjugate locus for them are empty.
\item
The \sL length maximizers coming to the zero-radius sphere are nonsmooth (concatenations of two smooth arcs forming a corner, nonstrictly normal extremal trajectories).
\item
\SL spheres and \sL distance are real-analytic if $d > 0$. 
\end{enumerate}
It would be interesting to understand which of these properties persist for more general \sL problems (e.g., for left-invariant problems on Carnot groups).

\bigskip
The authors thank A.A.Agrachev, L.V.Lokutsievskiy, and M. Grochowski for valuable discussions of the problem considered.
%===============Список литературы==================

%\newpage
\addcontentsline{toc}{section}{List of figures}
\listoffigures

%\newpage


\begin{thebibliography}{99}
\addcontentsline{toc}{section}{References}
\bibitem{versh_gersh}
A.M.~Vershik, V.Y.~Gershkovich, Nonholonomic Dynamical Systems. Geometry of
distributions and variational problems. (Russian) In: {\em Itogi Nauki i
Tekhniki:
Sovremennye Problemy Matematiki, Fundamental'nyje Napravleniya}, Vol.~16,
VINITI,
Moscow, 1987, 5--85. (English translation in: {\em Encyclopedia of Math. Sci.}
{\bf
16}, Dynamical Systems 7, Springer Verlag.)
\bibitem{jurd_book}
V.~Jurdjevic,
{\em Geometric Control Theory},
Cambridge University Press, 1997.
\bibitem{mont}
R. Montgomery, {\em A tour of subriemannian geometries, their geodesics and applications}, Amer. Math. Soc., 2002.
\bibitem{notes}
A. Agrachev, Yu. Sachkov, {\em Control theory from the geometric viewpoint}, 
Berlin Heidelberg New York Tokyo. Springer-Verlag, 2004.
\bibitem{ABB}
A. Agrachev, D. Barilari, U. Boscain, {\em A Comprehensive Introduction to
sub-Riemannian Geometry
from Hamiltonian viewpoint}, Cambridge University Press, 2019.
\bibitem{intro}
Yu. Sachkov, {\em Introduction to geometric control}, Springer, 2022.
 \bibitem{UMN}
Yu. Sachkov, Left-invariant optimal control problems on Lie groups: classification and problems integrable by elementary functions, {\em Russian Math. Surveys}, 77:1 (2022), 99--163
 \bibitem{groch2}
M. Grochowski,
 Geodesics in the sub-Lorentzian geometry. {\em Bull. Polish.
Acad. Sci. Math.}, 50 (2002).
 \bibitem{groch3}
M. Grochowski,
Normal forms of germs of contact sub-Lorentzian structures on
$\R^3$. Differentiability of the sub-Lorentzian distance. {\em J. Dynam. Control
Systems} 9 (2003), No. 4.
\bibitem{groch9}
M. Grochowski,
 Properties of reachable sets in the sub-Lorentzian geometry, {\em J. Geom. Phys.} 59(7) (2009) 885–900.
 \bibitem{groch11}
M. Grochowski, Reachable sets for contact sub-Lorentzian metrics on $\R^3$. Application to control affine systems with the scalar input, {\em J. Math. Sci.} (N.Y.) 177(3) (2011) 383–394.
 \bibitem{groch4}
M. Grochowski,
On the Heisenberg sub-Lorentzian metric on $\R^3$,
GEOMETRIC SINGULARITY THEORY,
BANACH CENTER PUBLICATIONS,   
INSTITUTE OF MATHEMATICS,
POLISH ACADEMY OF SCIENCES,
WARSZAWA, 
vol. 65,
2004.


 \bibitem{groch6}
M. Grochowski,
Reachable sets for the Heisenberg sub-Lorentzian structure on $\R^3$. An estimate for the distance function.
{\em Journal of Dynamical and Control Systems},
vol. 12,
2006,
2, 145--160.  



 
 \bibitem{chang_mar_vas}
D.-C. Chang, I. Markina and A. Vasil'ev, Sub-Lorentzian geometry on anti-de Sitter space,
{\em J. Math. Pures Appl.}, 90 (2008), 82--110.
 \bibitem{kor_mar}
A. Korolko and I. Markina, Nonholonomic Lorentzian geometry on some H-type groups, {\em J.
Geom. Anal.}, 19 (2009), 864--889.
 \bibitem{grong_vas}
E. Grong, A. Vasil’ev, Sub-Riemannian and sub-Lorentzian geometry on $SU(1, 1)$ and on its universal cover, {\em J. Geom. Mech.} 3(2) (2011) 225–260.


 \bibitem{groch_med_war}
M. Grochowski, A. Medvedev, B. Warhurst,
3-dimensional left-invariant sub-Lorentzian contact structures, 
{\em Differential Geometry and its Applications}, 49 (2016) 142--166
 \bibitem{vinberg}
H. Abels, E.B. Vinberg, On free two-step nilpotent Lie semigroups and inequalities between random variables, {\em J. Lie Theory}, 29:1 (2019), 79--87
\bibitem{PBGM}
L.S. Pontryagin,
V. G. Boltyanskii, R. V. Gamkrelidze, E.F. Mishchenko,
{\em Mathematical Theory of Optimal Processes},
New York/London. John Wiley \& Sons, 1962.
\bibitem{hak_ledon}
E. Hakavuori, E. Le Donne, Non-minimality of corners in subriemannian geometry, {\em Invent. Math.}, 206(3): 693--704, 2016.
\bibitem{lok_pod}
L.V. Lokutsievskiy, A.V. Podobryaev, Existence of length maximizers in sub-Lorentzian problems on nilpotent Lie groups, {\em in preparation}.
\bibitem{pop_sach}
A.Yu. Popov, Yu.L. Sachkov, Asymptotics of \sL distance at the Heisenberg group at the boundary of the attainable set, {\em in preparation}.
\end{thebibliography}
\end{document}